\documentclass[11pt]{article}

\usepackage{amsthm}
\usepackage{amsmath}
\usepackage{thmtools,thm-restate}
\usepackage{amssymb}
\usepackage{graphicx}
\usepackage[margin=2cm, includefoot]{geometry}
\usepackage{commath} 
\usepackage{mathtools}
\usepackage{xcolor}
\usepackage{color}
\usepackage{url}
\usepackage{algorithm}
\usepackage{bm}
\usepackage{verbatim}
\usepackage{subcaption}
\usepackage{enumitem}
\usepackage{algorithmic}
\usepackage{comment}
\usepackage{titlecaps}
\usepackage{comment}
\usepackage{tabularx}

\usepackage{tikz}
\usetikzlibrary{arrows.meta, positioning, shapes.geometric, calc,fit}
\usepackage{accents}

\usepackage{hyperref}   %cancel before uploading to arXiv?
\hypersetup{colorlinks=true, linkcolor=blue, filecolor=magenta, urlcolor=cyan, pdftitle={Hermite schemes over manifolds}, pdfauthor=author}

\usepackage{authblk}
\newcommand{\floor}[1]{\lfloor #1 \rfloor}

\newcommand{\R}{\mathbb{R}}
\newcommand{\N}{\mathbb{N}}

\newcommand{\Z}{\mathbb{Z}}
\newcommand{\PP}{\mathcal{P}}
\newcommand{\D}{\mathcal{D}}

\DeclareMathOperator{\s}{\mathcal{S}} 
\DeclareMathOperator{\LL}{\mathcal{L}}

\DeclareMathOperator{\muW}{\mu_{\scalebox{0.7}{W}}}

\DeclareMathOperator{\dist}{dist}

\DeclareMathOperator{\mudelta}{\mu_1}

\def\bsq#1{%both single quotes
\lq{#1}\rq}

\newcommand{\pa}[1]{\mathcal{P} \hspace{-2pt}\mathcal{A}\hspace{-1pt}\left({#1}\right)}

\newcommand{\rev}[1]{{\color{black}{#1}}}

\numberwithin{equation}{section}

\newtheorem{thm}{Theorem}
\numberwithin{thm}{section} 

\newtheorem{conj}[thm]{Conjecture}
\newtheorem{cor}[thm]{Corollary}
\newtheorem{prop}[thm]{Proposition}
\newtheorem{lemma}[thm]{Lemma}

\newtheorem{example}[thm]{Example}
\newtheorem{definition}[thm]{Definition}
\newtheorem{remark}[thm]{Remark}

\begin{document}

\title{Subdivision Schemes in Metric Spaces}% and their Analysis}

\author[1]{Nira Dyn\thanks{Email: niradyn@tauex.tau.ac.il}}
\author[1]{Nir Sharon\thanks{Corresponding author: nsharon@tauex.tau.ac.il}}

\affil[1]{School of Mathematical Sciences, Tel Aviv University, Tel Aviv, Israel}

%\date{\today}
\date{\vspace{-5ex}}
%% ====================================

\maketitle
\begin{abstract}
We develop a unified framework for nonlinear subdivision schemes on complete metric spaces (CMS). We begin with CMS preliminaries and formalize refinement in CMS, retaining key structural properties, such as locality. We prove a convergence theorem under contractivity and demonstrate its applicability. To address schemes where contractivity is unknown, we introduce two notions of proximity. Our proximity methods relate a nonlinear scheme to another nonlinear scheme with known contractivity, rather than to a linear scheme, as is common in the literature. Specifically, the first type of proximity compares the two schemes after a single refinement step and, as in the classical theory, yields convergence from sufficiently dense initial data. The proximity of the second type monitors alignment across all refinement levels and provides strong convergence without density assumptions. We formulate and prove the corresponding theorems and illustrate them with various examples, including schemes over metric spaces of compact sets in $\R^n$, the Wasserstein space, and a geometric Hermite metric space. \rev{We also include a non-stationary extension, giving a direct convergence criterion for level-dependent refinement operators in CMS in terms of summability of products of refinement factors.} These results extend subdivision theory beyond Euclidean and manifold-valued data to metric spaces.
\end{abstract}

\vspace{0.5em}
\noindent\textbf{Keywords:} subdivision schemes, complete metric spaces, convergence, proximity, smoothness

\vspace{1em}
\noindent\textbf{Mathematics Subject Classification (2020):} 65D17, 41A25, 41A63, 53C20

\vspace{1em}
\noindent\textbf{Communicated by Pencho Petrushev}

%% ====================================
\section{Introduction} 

% A general intro to subdivision schemes and their role as approximation techniques
Subdivision schemes are iterative algorithms that generate smooth curves or surfaces from discrete sets of control points in a Euclidean space~\cite{cavaretta1991stationary, dyn1992subdivision_book}. Originating in computer-aided geometric design (CAGD) and also investigated in Computer Graphics, these schemes refine an initial coarse shape into progressively smoother and more detailed forms. New points are computed at each iteration using local rules that blend neighboring values, approximating a continuous function or a geometric object. Subdivision schemes are widely applied in fields such as Animation, Computer Graphics, CAGD, and applied approximation ~\cite{liu2022review, dyn1992subdivision_book, conti2016approximation, cohen2002nonlinear}. The strength of subdivision schemes lies in the ability to balance computational simplicity with the production of visually and mathematically smooth results. Yet, the analysis of the limits of these schemes is considered challenging, and in particular, addressing the question of determining whether a given set of refinement rules leads to a converging process and what is the smoothness of the limits of such a process~\cite{dyn2002analysis}.

% The recent development of subdivision schemes for nonlinear domains
Recent developments in subdivision schemes have extended their applicability to nonlinear domains, such as manifolds, Lie groups, and other geometric spaces where traditional linear operations are not well-defined~\cite{wallner2005convergence, sharon2013approximation, dyn2017global, crouch1999casteljau, dyn2017manifold, noakes1998nonlinear, duchamp2018smoothing}. In these contexts, standard linear averaging used in classical schemes must be replaced with operations that respect the underlying geometry, such as geodesic interpolation or intrinsic means. This shift has enabled subdivision techniques to be applied to data arising in fields like computer vision, medical imaging, and robotics, where information often resides on curved spaces rather than in flat Euclidean settings~\cite{yassine2008subdivision, chen2021numerical, liu2022review}. These nonlinear subdivision schemes preserve critical structural properties of the data, such as distances, orientations, and constraints, making them essential tools for the accurate and efficient approximation of complex geometric and functional data~\cite{itai2013subdivision}.

% leaning toward metric spaces
Approximation techniques are increasingly being developed to address more general and complex nonlinear data settings beyond manifold-valued data. For example, some work has extended approximation frameworks to handle data in Wasserstein spaces, that is, metric spaces of probability measures~\cite{banerjee2025efficient}. Another example is set-valued data, where each data point is a nonempty, compact subset of a fixed Euclidean space rather than a scalar or a vector, which presents additional challenges due to the lack of a linear structure, prompting the design of specialized averaging and interpolation methods~\cite{dyn2014approximation, levin2025set}. These advancements reflect a broader trend towards designing subdivision and approximation schemes that can adapt to the intrinsic geometry and structure of diverse, nonlinear data types arising in modern applications across science and engineering. This paper aims to fill some of the theoretical gaps in the study of subdivision schemes for data in metric spaces, which includes all the above examples.

% The main contribution of this paper
Our analysis is set within the framework of subdivision schemes for data in a complete metric space (CMS) with an average. Our primary focus is on the question of convergence. The main contributions of this work are as follows. We begin by rigorously establishing the definitions of averaging and refinement sequences in a CMS with an average, laying the groundwork for a coherent extension of classical subdivision theory. We then investigate the fundamental properties of these schemes and formalize an appropriate notion of convergence in this general setting. The core of our analysis is divided into two main parts. In the first part, we prove a generalized convergence theorem that extends existing results from the manifold-valued case to the broader context of data from a CMS with an average. In the second part, we turn to proximity analysis, examining how to infer convergence and smoothness of one subdivision scheme from the properties of a related subdivision scheme.

% The proximity
Proximity analysis in subdivision schemes originated from the challenge of studying nonstationary schemes (where refinement rules can vary from level to level), by their ``proximity'' to stationary schemes (where the same refinement rules are applied at each refinement level)~\cite{dyn2003exponentials}. Note that in this ``proximity'' both schemes are linear. The concept was later extended to manifold-valued settings by comparing a linear scheme operating on points in a Euclidean space, such as a B-spline scheme, with its nonlinear counterpart operating on manifold data, showing how the nonlinear scheme inherits key convergence and smoothness properties from the linear scheme ~\cite{wallner2005convergence, duchamp2018smoothing}.

In this paper, we begin by introducing the {\bf first type of proximity}, which generalizes the above proximity idea by allowing both schemes to be nonlinear. In this context, we also treat the notion of a \rev{contractivity level} and use it with other quantifiers to provide sharper conditions on the input data for which convergence can be deduced. We prove a convergence theorem induced by proximity and provide an example of its application.

% The second type proximity
\rev{We introduce a new notion of proximity, which we call {\bf proximity of the second type}. This notion appears here for the first time. Recall that} the first type of proximity compares two subdivision schemes after a single refinement step and uses this comparison to infer their similar limiting behavior. In contrast, the second notion of proximity examines how one scheme consistently follows another scheme at all refinement levels, with the difference between them shrinking at each step. Unlike the first type, the form of the second type of proximity depends on the input data and is not symmetric--only one scheme is required to approximate the other. Crucially, it guarantees a stronger form of convergence, valid from any initial data. 

Additionally, the second type has a clear notion of order that in Euclidean spaces directly reflects the smoothness of the limit: zero order ensures convergence, while first order yields $C^1$ smoothness, whenever the scheme to which the analyzed scheme is in proximity with, has these properties. In the absence of a natural notion of smoothness in \rev{CMSs}, our smoothness analysis is confined to Euclidean spaces. We establish both convergence and smoothness results and demonstrate their application by examples.

\rev{Finally, we return to the classical motivation for proximity analysis, namely non-stationary subdivision schemes. In Section~\ref{sec:nonstationary} we formulate a CMS-valued non-stationary setting and prove a direct convergence criterion for such schemes. The criterion is expressed in terms of level-dependent refinement factors and block displacement estimates. Under uniform bounds on the refinement levels, the block displacement-safe constants, and the refinement factors, convergence follows from the summability of the corresponding products of refinement factors.}

% The paper is organized as follows
The paper is organized as follows. Section~\ref{sec:perlimin} presents a brief background survey that provides the necessary context for our work. The following Section~\ref{sec:CMS_perlimn} completes this survey with the definition of average in CMS, accompanied by examples. In Section~\ref{sec:main_convergence}, we introduce the foundational definitions and results related to refinement in complete metric spaces with an average. Section~\ref{sec:proximity1} is devoted to the first type of proximity analysis, while Section~\ref{sec:proximity2} introduces and explores the second type of proximity. In Sections~\ref{sec:proximity1} and ~\ref{sec:proximity2}, examples in various metric spaces are provided. \rev{Section~\ref{sec:nonstationary} is devoted to the definitions and convergence results concerning non-stationary schemes.}

A concise overview of the paper is provided as a roadmap in Figure~\ref{fig:org-map-new}.

\begin{figure}[h!]
\centering
\begin{tikzpicture}[
  >=Latex,
  every node/.style={font=\small},
  box/.style={draw, rounded corners=2pt, align=left, inner sep=2.5pt},
  wide/.style={box, text width=.98\textwidth},
  col/.style={box, text width=.47\textwidth},
  line/.style={-{Latex}}
]

% Top: Sec 2
\node[wide] (s2) {\textbf{Sec.~2 --- A short review on linear subdivision in $\mathbb{R}^n$.}};

% Middle 1: Sec 3
\node[wide, below=4mm of s2] (s3) {
\textbf{Sec.~3 --- CMS preliminaries:}\\
binary averages; intrinsic vs.\ non-intrinsic; brief examples.
};

% Middle 2: Sec 4
\node[wide, below=4mm of s3] (s4) {
\textbf{Sec.~4 --- Refinement \& convergence in CMS\rev{s}:}\\
piecewise average interpolant; refinement and parameters; limit curve;
main convergence theorem; applications to elementary schemes; approximation theorem.
};

% Bottom row: Sec 5 and Sec 6
\node[col, anchor=north west] (s5) at ([yshift=-8mm]s4.south west) {
\textbf{Sec.~5 --- First type proximity:}\\
definition; weak convergence result; application via example.
};

\node[col, anchor=north east] (s6) at ([yshift=-8mm]s4.south east) {
\textbf{Sec.~6 --- Second type proximity:}\\
definition; strong convergence result; $C^1$ smoothness in $\mathbb{R}^n$
for order $m{=}1$; $C^m$ conjecture on order $m$; $C^1$ example.
};

% Final row: Sec 7
\node[inner sep=0pt, fit=(s5)(s6)] (s56) {};
\node[wide, below=6mm of s56] (s7) {
\textbf{\rev{Sec.~7 --- Non-stationary subdivision schemes:}}\\
\rev{level-dependent refinement operators; refinement factors;
direct convergence criterion under summability of refinement-factor products.}
};

% Arrows
\draw[line] (s2) -- (s3);
\draw[line] (s3) -- (s4);
\draw[line] (s4.south) -- ++(0,-2mm) -| (s5.north);
\draw[line] (s4.south) -- ++(0,-2mm) -| (s6.north);
\draw[line] (s5.south) -- ++(0,-2mm) -| ([xshift=-18mm]s7.north);
\draw[line] (s6.south) -- ++(0,-2mm) -| ([xshift=18mm]s7.north);

\end{tikzpicture}
\caption{The paper's roadmap.}
\label{fig:org-map-new}
\end{figure}
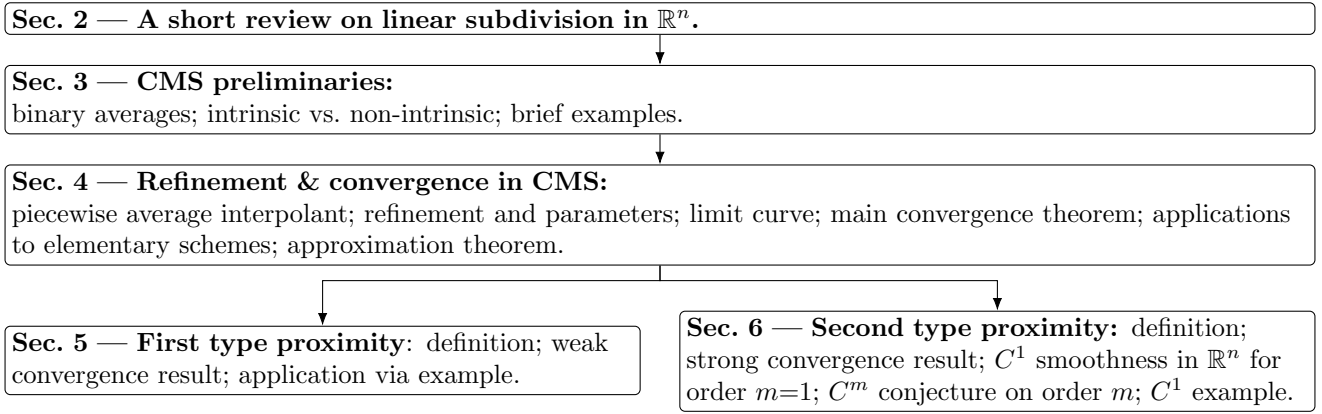

%===============================================================
\section{A quick recap: binary linear subdivision schemes in Euclidean spaces} \label{sec:perlimin}

Given a sequence $V^0=\{v^0_i\}_{i\in\Z}$ of  \rev{bounded} vectors in a Euclidean space\rev{, such that $\sup_{i\in\Z} \norm{v^0_i} <\infty$}. A binary, linear subdivision scheme generates curves by repeatedly applying a refinement operator $S_{\rev{\textbf {a}}}$ of the form 
\begin{equation}\label{eq:refinementoper}
     V^{k+1} = S_{\rev{\textbf {a}}}V^k,\quad v^{k+1}_i=\sum_{j\in \Z}a_{i-2j}v^k_j,\quad i\in\Z, \quad k\in\Z_+ .
\end{equation}
%\end{document}
Here $\rev{\textbf {a}}=\{a_i\}_{i\in\Z}$ is a sequence of real numbers of {\bf finite support}, which is called the mask of the refinement operator $S_{\rev{\textbf {a}}}$ and $V^k$ is the sequence of vectors generated after $k$ refinements of $V^0$ by $S_{\rev{\textbf {a}}}$, namely $V^k=S_{\rev{\textbf {a}}}^kV^0 . $
%\end{document}

%Here $\rev{\textbf {a}}=\{a_i\}_{i\in\Z}$ is a sequence of real numbers of {\bf finite support}, which is called the mask of the refinement operator$S_{\rev{\textbf {a}}}, and $V^k$ is the sequence of vectors generated after $k$ refinements of $V^0$ by $S_{\rev{\textbf {a}}}$, namely $V^k=S_{\rev{\textbf {a}}}^kV^0$.
Note that in \eqref{eq:refinementoper}, there is one refinement rule for the vectors with even indices and another rule for the vectors with odd indices. In addition, \eqref{eq:refinementoper} implies that every component of the vectors of  $V^k$ \rev{is} refined by the same refinement, independently of the other components. 

%\end{document}
The sequence $V^k=\{v^k_i\}_{i\in\Z}$ is associated with a sequence of parameters $T^k=\{t^k_i\}_{i\in\Z}$ for all $k\in \Z_+$. For binary linear subdivision schemes generating curves, it is customary to use $T^0 = \Z$, and $T^k = (\frac{1}{2})^k \Z = \frac{1}{2}T^{k-1}$, known as the primal parameterization, or $T^0=\Z$ and $T^k = (\frac{1}{2})^k \Z + \sum_{\ell=2}^{k+1}(\frac{1}{2})^{\ell} = \frac{1}{2}\left(T^{k-1}+\frac{1}{2}\right)$, known as the dual parameterization.

To define { \bf convergence} of a subdivision scheme, we introduce the sequence $\{ {\mathcal L}_k(t)\}_{k\in\Z_+}$, where ${\mathcal L}_k(t)$ is a
vector-valued piecewise linear interpolants of the sequence of points $\{(t^k_i,v^k_i)\}_{i\in\Z} $. Thus 
$${\mathcal L}_k (t)=(1-w)v_i^k+w{v^k_{i+1}},\quad w=\frac{t-t^k_i}{t^k_{i+1}-t_i^k},\quad t\in[t_i^k,t_{i+1}^k],\quad i\in \Z. $$ 
The limit, in the uniform norm, of the sequence $\{ {\mathcal L}_k(t)\}_{k\in\Z_+}$, if it exists, is defined as
the limit curve generated by the subdivision scheme.

\rev{A linear subdivision scheme that converges from any initial sequence $V^0$ of bounded vectors is termed ``convergent''.} Determining if a subdivision scheme is convergent, given by the mask of its refinement operator, is a central issue in the theory of linear subdivision schemes and is solved completely for schemes generating curves, see, e.g.~\cite{dyn1992subdivision_book}. This issue is also central to the present paper, where we study subdivision schemes in complete metric spaces with an average.
%\end{document}

As an example, we consider the canonical family of linear subdivision schemes generating spline curves.% This is a canonical family of schemes whose limits are known explicitly.
The subdivision scheme based on the refinement operator with the mask 
%operator $S_{\rev{\textbf {a}^{[n]}}}$ with the mas
$\rev{a_i=}(\frac{1}{2})^n\binom{n+1}{i}, i=0,\ldots n+1$, is convergent for any $n\in\N$ and generates limit curves which are splines of degree $n$ with integer knots having continuous derivatives up to order $n-1$. 
%\end{document}
The scheme corresponding to $n=1$ \rev{ has the mask $(\frac{1}{2},1,\frac{1}{2})$ and} the two refinement rules
\begin{equation}\label{eqn:piecewiselin}
    v^{k+1}_{2i} = v^k_i,\quad v^{k+1}_{2i+1}= \frac{1}{2}(v^k_i + v^k_{i+1}).
    \end{equation}
 This scheme is an elementary scheme with respect to the arithmetic average (see \eqref{eqn:elementary_scheme}), and its corresponding parameterization is the primal parameterization, $t_j^k=j\cdot 2^{-k}$. It is easy to see its convergence by observing that all the corresponding ${\mathcal L}_k,\  k\in \rev{{\mathbb Z}_+} $
are equal. This is the case since the points at refinement level $k+1$ lie on ${\mathcal L}_k$, for all $k\in {\mathbb Z}_+$.
 
 The scheme corresponding to $n=2$ has  \rev{the mask $(\frac{1}{4}, \frac{3}{4}, \frac{3}{4}, \frac{1}{4}) $ and the two} refinement rules
\begin{equation}\label{eqn:cc}
    v^{k+1}_{2i} =\frac{3}{4}v^k_i + \frac{1}{4}v^k_{i+1},\quad v^{k+1}_{2i+1}= \frac{1}{4}v^k_i + \frac{3}{4}v^k_{i+1}.
    \end{equation}
The corresponding parameterization is the dual parameterization. This is the well-known Chaikin’s corner-cutting scheme~\cite{chaikin1974algorithm}. This scheme can also be computed by applying, in each refinement level, first the elementary scheme based on the arithmetic average, and then a step of inserting the arithmetic average of any pair of consecutive points between the two points~\cite{lane1980theoretical}.

\rev{A scheme as in equation \eqref{eq:refinementoper}, where the same refinement operator is applied in all refinement levels, is termed \textbf {stationary scheme}. A \textbf {non-stationary} subdivision scheme is defined by a sequence of refinement operators $ \{S_{\textbf {a}^{[j]}}\}_{j\in\Z_+},$ such that at refinement level $k\in\N$,
$ V^{k+1}= S_{\textbf {a}^{[k]}}V^k.$ 

Non-stationary subdivision schemes are fundamental in geometric modeling due to their enhanced flexibility and adaptability, allowing them to generate a broader class of functions. For example, their ability to reproduce exponential polynomials makes them well-suited for accurately modeling shapes such as conics and spirals, capabilities that lie beyond those of stationary schemes, see e.g.,~\cite{dyn2003exponentials}. This classical non-stationary setting motivates the CMS-valued extension
developed in Section~\ref{sec:nonstationary}.
}

%% ===========================================================
\section{CMS Preliminaries: Averages, Intrinsic Averages and Examples} \label{sec:CMS_perlimn}

This section lays out the basic notions we need in complete metric spaces prior to studying refinement. We introduce the idea of an ``average'' between two elements of the space and discuss the relationship between the average and the metric. We illustrate the ideas through three brief examples: compact sets with the Hausdorff metric, probability measures in the Wasserstein space, and a geometric Hermite setting in $\mathbb{R}^n\times S^{n-1}$. These preliminaries set the stage for the refinement and convergence theory of subdivision schemes developed in the next section.

%----------------------------------------------------------
\subsection{Averages in metric spaces}

We consider a metric space  $X=(\Omega,d)$, where $\Omega$ is the collection of elements of $X$ endowed with a metric $d$. \rev{
Before introducing averages in metric spaces, we clarify their role in the present work. The primary purpose of the average is to provide an analog of piecewise linear interpolation for discrete sequences of points in a complete metric space. This construction serves as the foundation for our notion of parametric convergence, which we develop in Section~\ref{sec:main_convergence} using piecewise average interpolants. In addition, the same averaging operation is used in several examples to define subdivision rules; see, for example, the elementary refinement rule~\eqref{eqn:elementary_scheme} and the corner-cutting type scheme~\eqref{eqn:additional_round}.
}

The average defined next extends the notion of a binary average of numbers, such as the arithmetic average $ A_\omega(x,y)=(1-\omega)x + \omega y$, or the harmonic mean $ A_\omega(x,y)= \left( (1-\omega)\frac{1}{x} + \omega \frac{1}{y} \right)^{-1}$, for $x,y\in {\mathbb R}$, to elements of $X$.

\begin{definition} \label{def:average}
    Let $X=(\Omega,d)$ be a metric space. A map $A_\omega \colon \Omega \times \Omega \to \Omega$ is called a binary average (or, for short, average) in $X$ if it satisfies:
    \begin{enumerate}
    \item \label{def:item:continuous} $A_\omega$ is defined for any $\omega \in [0,1]$ and is continuous with respect to $\omega$ there.
    \item \label{def:item:endpoints} (End points interpolation) For any $x,y \in \Omega$, $A_0(x,y) = x$ and $A_1(x,y) = y$.
    \item \label{def:item:identity} (Identity on the diagonal) For any $\omega \in [0,1]$ and $x \in \Omega$, $A_\omega(x,x) = x$. %$\lim_{y \to x} A_\omega(x,y) = x$.
   \item (Symmetry) For any $\omega \in [0,1]$ and $x,y \in \Omega$, $A_\omega(x,y) = A_{1-\omega}(y,x) $
   \item(Boundedness) Under the conditions of the previous item $$\max \{ d(x,A_\omega(x,y)),d(y,A_\omega(x,y)) \} \le d(x,y).$$ 
    \end{enumerate}
\end{definition}
\begin{example}
\rev{
A useful source of binary averages is given by weighted
Fr\'echet means. Let \(X=(\Omega,d)\) be a metric space and, for \(x,y\in\Omega\) and \(\omega\in[0,1]\), consider the functional
\[
    F_{x,y,\omega}(a)
    =
    (1-\omega)d^2(x,a)+\omega d^2(a,y),
    \qquad a\in\Omega .
\]
When the minimization problem
\begin{equation} \label{eqn:frechet_def}
     A_\omega(x,y)=\arg\min_{a\in\Omega} F_{x,y,\omega}(a)
\end{equation}

has a unique solution for every \(x,y\in\Omega\) and every
\(\omega\in[0,1]\), and when, for every fixed \(x,y\in\Omega\), the map
\[
    \omega\mapsto A_\omega(x,y)
\]
is continuous on \([0,1]\), then \eqref{eqn:frechet_def} defines a binary average in the sense of Definition~\ref{def:average}.

Indeed, the endpoint conditions follow from the cases \(\omega=0\) and \(\omega=1\), and the identity on the diagonal follows by taking \(x=y\). The symmetry property follows from the identity
\[
    F_{x,y,\omega}(a)=F_{y,x,1-\omega}(a).
\]
Finally, for \(0<\omega<1\), comparison of the minimizer with \(a=x\) gives
\[
    d(A_\omega(x,y),y)\le d(x,y),
\]
and comparison with \(a=y\) gives
\[
    d(x,A_\omega(x,y))\le d(x,y).
\]
Thus, the boundedness condition in Definition~\ref{def:average} is satisfied. The assumed
continuity of \(\omega\mapsto A_\omega(x,y)\) gives the required continuity with respect to the averaging parameter.

The above assumptions are essential. In a general complete metric space, the minimizer may fail to be unique, may fail to exist, or may not depend continuously on \(\omega\). For example, in a discrete metric space with more than one point, the minimizer is not unique for \(\omega=1/2\). Similarly, on spaces with non-unique geodesics, such as antipodal points on a sphere, there is no canonical continuous choice of a midpoint.

In other familiar settings, the weighted Fr\'echet minimizer is well posed, for example, in Euclidean spaces, Hadamard spaces, and suitable geodesically convex subsets of Riemannian manifolds.
}
\end{example} 

\begin{definition} \label{def:metric+intrinsic}
An average $A_\omega$ has the \textbf{metric property} if in addition to the five items in Definition \ref{def:average} it satisfies for any $x,y\in\Omega$, 
\begin{equation} \label{defd=metric property}
 d(x,A_{\omega}(x,y)) = \omega d(x,y),\, \quad \omega \in [0,1].
\end{equation}
We call an average with the metric property \textbf{intrinsic average}. 
\end{definition}

 Two canonical examples of an intrinsic average are the arithmetic mean in $\R^n$ and the geodesic average in a complete Riemannian manifold, see e.g.,~\cite{dyn2017manifold}. 
\begin{remark} \label{rem:metric property}
Two remarks about the metric property:
    \begin{enumerate}
        \item The metric property~\eqref{defd=metric property}, and the symmetry property of Definition~\ref{def:average}, imply that 
   $$ d(y,A_{\omega}(x,y))=(1-\omega)d(x,y).$$ This together with~\eqref{defd=metric property} leads to the following {\bf intermediate value property}:
        \begin{equation} \label{eqn:intermidiate}
        d\bigl(x,A_\omega(x,y)\bigr)+d\bigl(A_\omega(x,y),y\bigr)=d(x,y), \qquad \omega\in[0,1].
        %d(y,A_{\omega}(x,A_{\omega}(x,y))+d(A_{\omega}(x,y),y) = d(x,y),\,\quad \omega \in [0,1].
        \end{equation}
        \item 
        For $x=y$, \eqref{defd=metric property} reduces to identity on the diagonal, $x=A_{\omega}(x,x)$ (the third item in
     Definition~\ref{def:average}). 
    \end{enumerate}
\end{remark}

We present two examples of metric spaces with an intrinsic average, and one example of a metric space with an average that is not intrinsic. 
%Currently, it is unknown to the authors of~\cite{hofit2023geometric}, whether by changing the metric it is possible to obtain a CMS  such that this average is intrinsic there. 

\begin{example}\label{example:compact sets}
Consider the metric space of nonempty, compact subsets of 
$\R^n$, endowed with the Hausdorff metric. We denote the collection of elements of this metric space by ${\cal K}_n$. Then, define for any set $B \in {\cal K}_n$,
   \[ \dist(a, B) = \min_{b \in B} \norm{a-b} , \quad a \in \R^n . \]
For any two sets $A, B \in {\cal K}_n$, we define the set of all their \textbf{metric pairs} as
   \begin{equation*}
       \Pi(A,B)=\bigg\lbrace (a,b) \colon \dist(a, B) = \norm{a-b} \, \text{ or } \,  \dist(b,A)=\norm{a-b} \bigg\rbrace.
   \end{equation*}
  Their Hausdorff distance $d_H(A,B)$, can be written in terms of their metric pairs,
  as 
   \begin{equation*}
   d_H(A,B) = \max \bigg\lbrace \norm{a-b} \colon (a,b) \in \Pi(A,B) \bigg\rbrace .
   \end{equation*}
Note that in this example, we can use the quantifiers 'min' and 'max' since the sets in ${\cal K}_n$ are compact. Moreover, the metric space $X=({\cal K}_n, d_H)$ is complete~\cite[Chapter 4]{rockafellar2009variational}, an important property when convergence of subdivision schemes is considered.
   
An intrinsic average (average with the metric property~\eqref{defd=metric property}) in  $({\cal K}_n, d_H)$ is,
   \begin{equation*}
  {\mathcal A}_{\omega}(A,B) = \bigg\lbrace \omega a+(1-\omega)b \colon (a,b) \in \Pi(A,B) \bigg\rbrace .
   \end{equation*}
 This average was designed in~\cite{artstein1989piecewise} to have the metric property. It is used there to approximate univariate set-valued functions (univariate functions with values in ${\cal K}_n$) by  ``piecewise linear approximants,'' built upon this average. In later works, this average is called {\bf metric average}, and is used for approximating set-valued functions by other types of classical approximation operators, see, e.g.,~\cite{dyn2014approximation}.
\end{example}
\begin{example} \label{example:Wasserstein}
The second example is in the space of probability measures on $\mathbb{R}^d$, $\mathcal{P}(\mathbb{R}^d)$, with finite $p$-th moment, denoted by 
\begin{equation}~\label{eqn:wasserstein_space}
    \mathcal{P}_p(\mathbb{R}^d) = \bigg\{ \mu\in\mathcal{P}(\mathbb{R}^d) \; \big| \;\int_{\mathbb{R}^d}\|x\|^p d\mu(x) < \infty \bigg\}.
\end{equation}
We define the metric via the functional $\mathcal{J}_p$ on a probability measure $\gamma$ over $\mathbb{R}^d \times \mathbb{R}^d$, given by
\begin{equation}~\label{eqn:wasserstein_functional}
    \mathcal{J}_p(\gamma) = \int_{\mathbb{R}^d\times \mathbb{R}^d}\|x-y\|^p d\gamma(x,y).
\end{equation}
In optimal transport terms, the integrand of $\mathcal{J}_p$ is the cost function; the Wasserstein distance is
\begin{equation}~\label{eqn:Wasserstein_distance}
    W_p(\mu, \nu) = \left( \min_{\gamma\in\Pi(\mu, \nu)}\mathcal{J}_p(\gamma) \right)^{\frac{1}{p}},
\end{equation}
where $\Pi(\mu,\nu)$ is the set of all joint measures with marginals $\mu$ and $\nu$,
\begin{equation}~\label{eqn:joint_measures}
    \Pi(\mu,\nu)=\{\gamma\in\mathcal{P}(\mathbb{R}^d\times\mathbb{R}^d)\;|\;\; (\pi^x)_\#\gamma=\mu,\;\;(\pi^y)_\#\gamma=\nu\}.
\end{equation}
Here, $\pi^x$ and $\pi^y$ are the projections onto the $x$- and $y$-coordinates, with $\#$ denoting pushforward. The set $\Pi(\mu,\nu)$ is nonempty, as it includes $\mu \times \nu$.

The space $\mathcal{P}_p(\mathbb{R}^d)$ equipped with $W_p$ is called \textbf{Wasserstein space}. The Wasserstein space admits \rev{ the Otto–Lott infinite-dimensional Riemannian formalism on suitable smooth subspaces~\cite{LottWassersteinRiemannian2008}. In the present work, however, we use only the metric/geodesic structure needed to define the average.} The average is then derived from the \textbf{geodesic} in the Wasserstein space. Specifically, when $\gamma$ of~\eqref{eqn:Wasserstein_distance} is obtained, the average between $\mu_{0}$ and $\mu_{1}$ is given by the pushforward,
\begin{equation*}
    {\mathcal A}_{\omega}(\mu_{0},\mu_{1}) = (\pi_\omega)_{\#}\gamma, \quad \omega \in [0, 1].
\end{equation*} 
Where we define $\pi_\omega(a,b) = (1-\omega)a +\omega b$ for any $a,b \in  \R^d$. The resulting ${\mathcal A}_{\omega}(\mu_{0},\mu_{1})$ is the shortest path between $\mu_{0}$ and $\mu_{1}$ in the Wasserstein space, see, e.g., \cite{santambrogio2015optimal}.
\end{example}
\begin{example} \label{example:hermiteaverage}
In the third example, we consider the domain \(\Omega = \mathbb{R}^n \times \mathbb{S}^{n-1}\). As a product of manifolds, this domain is a manifold where we interpret each point as a pair $(p,v)$, consisting of a point in the space and a direction associated with the point. 

In the geometric setting of the above $\Omega$, given two different pairs, $(p_0,v_0)$ and $(p_1,v_1)$, we consider a cubic curve that interpolates both points and tangent directions, described as follows. First, we define four control points: $\mathbf{P}_0 = p_0, \mathbf{P}_1 = p_0+c v_0, \mathbf{P}_2 = p_1-c v_1,$ and $\mathbf{P}_3= p_1$. Here $c>0$ is fixed. Then, the cubic B\'ezier curve interpolating these four control points is
\[
\mathbf{B}(t) = (1 - t)^3 \mathbf{P}_0 + 3(1 - t)^2 t \mathbf{P}_1 + 3(1 - t) t^2 \mathbf{P}_2 + t^3 \mathbf{P}_3, \quad t \in [0, 1] .
\]
One can easily verify that this curve passes through the averaged two points and the directions $v_0, v_1$ are the tangents of the curve at $p_0,p_1$, respectively. 

Finally, we define the average ${\mathcal A}_{\omega}$ as the pair consisting of the value of the B\'ezier curve at $t=\omega$ and the tangent of the  B\'ezier curve there. In~\cite{hofit2023geometric}, it is proven that ${\mathcal A}_{\omega}$ is an average. %, with a specific choice of $c=\frac{\norm{p_{1}-p_{0}}}{3\cos^2{\theta/4}}$, with $\theta = \arccos(v_0 \cdot \frac{v_1-v_0}{\norm{v_1-v_0}})+\arccos(v_1 \cdot \frac{v_1-v_0}{\norm{v_1-v_0}})$. 
For completeness, in the case of two identical pairs, we define this pair as the average. This average plays a central role in solutions for the problem of geometric Hermite interpolation~\cite{hofit2023geometric, vardi2024hermite}. This problem is described later in Section~\ref{subsection:application_first_type_proximity}.

Interestingly, although the construction of ${\mathcal A}_{\omega}$ has proven useful in certain geometric contexts, to the best of our knowledge, the existence of a metric on $\Omega$ that renders this average the metric property has not been established, see, e.g.,~\cite{hofit2023geometric}. 

\end{example}

%% ===========================================================

\section{Subdivision in CMS: Refinement, Limit Curves, and Convergence} \label{sec:main_convergence}

We now turn to the study of refinement and convergence in complete metric spaces (CMS), building on the preliminaries in Section~\ref{sec:CMS_perlimn}. We begin by introducing the necessary definitions and new concepts concerning subdivision schemes in a CMS. With these foundations in place, we proceed to establish several results on the convergence of such schemes, along with an additional result on their approximation quality. 
   
We start with an extension to the definition of the piecewise linear interpolant in Euclidean spaces~\cite{dyn1992subdivision_book} and the piecewise geodesic interpolant in manifolds~\cite{dyn2017manifold} to a metric space with an average. For that we consider a sequence of points $\PP=\{p_j\}_{j\in \Z}\subset \Omega$, and a sequence of parameters $T=\{t_j\}_{j\in \Z}\subset \R$, satisfying $t_j<t_{j+1}$ for $j\in \Z$.
\begin{definition} \label{def:piecewise_average}
 Let $Q = \{(t_j,p_j) \}_{j \in \Z}$, and define the following curve in $\Omega$,
    \begin{equation*}
        \pa{t} = \pa{A_w, Q; t} \coloneq A_{\omega}(p_j,p_{j+1}), \quad \omega = \frac{t-t_{j}}{t_{j+1}-t_j}, \quad  t \in [t_j, t_{j+1}] , \quad j \in \Z.
    \end{equation*}
    We call $\pa{A_w, Q; t}$ the \textbf{piecewise average interpolant} of $Q$. 
\end{definition}
By Definition~\ref{def:average}, items \ref{def:item:continuous} and \ref{def:item:endpoints}, the piecewise average interpolant is continuous and satisfies the interpolation conditions $\pa{t_j} = p_j,\quad j \in \Z$.

%----------------------------------------------------------\subsection{Refinement and subdivision in complete metric spaces with an average}

Here, we introduce the notion of refinement of a sequence consisting of elements of a CMS with an average $A_\omega$, and define the convergence of a sequence of sequences of type $Q$ to a curve over $\Omega$. We begin with the definition of subdivision schemes in a CMS with an average. 

From here on,  for simplicity, we assume that any CMS considered has an average, and we denote by $(\PP)_j$ the $j$-th element in the sequence $\PP$. \rev{In addition, we restrict the presentation to two types of 
parameterizations, primal and dual. In the primal parameterization $T^k = (\frac{1}{2})^k \Z = \frac{1}{2}T^{k-1}$, and in the dual  parametrization,  $T^0=\Z$ and $T^k=\frac{1}{2}\left(T^{k-1}+\frac{1}{2}\right)$. The dual parameterization is only used in Section~\ref{sec:proximity2}.} 

\begin{definition}[\rev{Locality of }refinement operator] \label{def:refinement}
\rev{Let $\PP  \in \Omega^\Z $ be a sequence} associated with a sequence of parameters \rev{$T$}. A binary refinement operator $S \colon \Omega^\Z \to \Omega^\Z$ generates a new sequence $S(\PP) \in \Omega^\Z$, such that $S$ is local in the sense that there exists $ M_S \in \N$ so that $S(\PP)_{2j}$ and $S(\PP)_{2j+1}$ are independent of $\PP_\ell$ for $\ell < j -  M_S $ and $ \ell > {j+ M_S}$. We call the constant $ M_S$ {\bf the locality range} of the operator $S$.
%The sequence $S(\PP)$ is associated with a refined sequence of parameters $R(T)$ \rev{which is either $\frac{1}{2}T$ in the primal parameterization or $\frac{1}{4} + \frac{1}{2}\T$ {\color {red} ???}in the dual parametrization. The choice of the parameterization depends on the geometry of the refinement.} 
\end{definition}
\begin{definition}[Subdivision scheme] \label{def:subdivision}
\rev{Let $\PP^0$ be a sequence in $\Omega^\Z$, associated with a parameterization $T^0$}. A binary subdivision scheme is the repeated application of the refinement operator $S \colon \Omega^\Z \to \Omega^\Z $,
\begin{equation*}
    \PP^{k} = S(\PP^{k-1}) , \quad k \in \N  ,
\end{equation*}
where each $\PP^k$ is associated with \rev{the parameterization} $T^k=\{t^k_j\}_{j\in\Z}$ in the sense that the point $(\PP^k)_j=p^k_j$ is associated with the parameter $t^k_j$.
\end{definition}

A basic binary subdivision scheme that exists in any CMS with an average $A_w$ is defined by the refinement rule
\begin{equation}   \label{eqn:elementary_scheme}
 S(\PP)_{2j}=p_j,\quad S(\PP)_{2j+1}=A_\frac{1}{2}(p_j,p_{j+1}),\quad j\in\Z.
 \end{equation}
We term such a scheme {\bf elementary scheme}. The parameterization of the elementary scheme is the primal parameterization, namely, \rev{$p_j^k$ is associated with $t_j^k=j2^{-k}$. This choice of the parameterization guarantees that a point in any $\PP^k$, which is inserted at some level, has the same associated parameter value in $\R$ at all refinement levels.}
\rev{
\begin{remark}[Finite dependence under iteration] \label{rem:locality}
    Let \(S\) be a local binary refinement operator with locality range \(M_S\)
in the sense of Definition 4.2. Then, for every \(k\in\mathbb N\),
every \(j\in\mathbb Z\), and every \(r=0,\ldots,2^k-1\), the point
\[
    (S^k(\PP))_{2^k j+r}
\]
depends only on finitely many entries of the initial sequence \(\PP\). More precisely, a simple induction shows that it depends only on entries
\(p_\ell\) with
\[
    j+\left\lfloor \frac{r}{2^k}-M_S^{[k]}\right\rfloor
    \le \ell \le
    j+\left\lfloor \frac{r}{2^k}+M_S^{[k]}\right\rfloor,
    \qquad
    M_S^{[k]}=\sum_{n=0}^{k-1}\frac{M_S}{2^n}.
\]
In particular, since \(M_S^{[k]}<2M_S\), the dependence range is contained in
\[
    j-2M_S \le \ell \le j+2M_S .
\]
This is the finite-dependence property of the refined data. Note that the factor \(2\) in this finite-dependence bound is sharp in general; for instance, the linear corner-cutting scheme (see Section~\ref{sec:perlimin}) has \(M_S=1\), as some refined points depend on initial data two indices away, see, e.g.,~\cite{JEONG2021113446}.
\end{remark}
 }

The convergence of a sequence of refined sequences associated with sequences of parameters is defined as follows:
\begin{definition}[The limit curve]
\rev{Denote $Q^k=\{(t_j^k,p^k_j)\}_{j\in \Z}$, $k \in \N$ and let
$$\bigg \lbrace f^k(t)= \pa{A_w,Q^k; t} \bigg\rbrace_{k \in \N} , $$ 
be the sequence of piecewise average-interpolants of $\{Q^k\}_{k \in \N}$. We call the subdivision sequence $\{S^k(\PP)\}_{k \in \N}$ converging if the sequence ${f^k(t):\mathbb{R}\to\Omega}$ converges uniformly in the metric 
    \begin{equation}  \label{eqn:metric-functions}
    d_\infty (f,g)=\sup\bigg\{d(f(t),g(t)) \colon t\in \R \bigg\} ,
    \end{equation}
    with $f,g$ any two functions of $t\in\R$ with values in $\Omega$. 
We denote the limit function by $Q^\infty(t)$, $t\in \R$. We call a subdivision scheme $S$ convergent in $\D\subset\Omega^{\mathbb{Z}}$ if $S(\D)\subset \D$ and the sequence $\{S^k(\PP)\}_{k \in \N}$ created by it converges for any $\PP\in \D$.}
\end{definition}
\begin{remark}
We observe the following about the limit curve:
\begin{enumerate}
    \item (Continuity)
    From Definition~\ref{def:piecewise_average}, $f^k(t) = \pa{A_w,Q^k; t}$ is continuous in $\R$ for any $k \in \N$. Therefore, in any \rev{compact interval} $I \subset \R$, the limit $Q^\infty$ is continuous. The latter is true since the convergence is with respect to the sup metric~\eqref{eqn:metric-functions}, so the convergence is uniform in any \rev{compact interval}. 
    \item \rev{(Finite dependence) For each fixed refinement level \(k\), the interpolant \(f^k(t) = \pa{A_w,Q^k; t}\) depends only on finitely many entries of the initial sequence \(\PP\). More precisely, this follows from the finite-dependence property of the iterates \(S^k(\PP)\) described in Remark~\ref{rem:locality}, together with the fact that \(f^k\) is built from two neighboring level-\(k\) points. Consequently, whenever the subdivision sequence converges, the limit curve has the corresponding finite dependence on the initial data.}
\end{enumerate}
\end{remark}

%----------------------------------------------------------
\subsection{Main convergence theorem}

For $\PP = \{p_j\}_{j \in \Z}$  a sequence in $\Omega^\Z$, we define 
\begin{equation} \label{eqn:delta_def}
    \delta (\PP) = \sup_{j \in \Z} d(p_j,p_{j+1}) .
\end{equation}
This bound for sequences of points in $\Omega^\Z$ is central to the definition of the next two fundamental notions of convergence. \rev{Throughout the convergence results, we consider initial sequences \(P\) satisfying \(\delta(P)<\infty\).}
\begin{definition}(Contractivity)
\label{def:contractivity}
A subdivision scheme $S$ is termed \textbf{contractive} in $\D \subset \Omega^\Z$ if there exists $\mu \in (0,1)$ and $L \in \N$ such that, 
\begin{equation} \label{eqn:contractivity_factor}
     \delta (S^L(\PP)) \le \mu \delta( \PP) , \quad \PP \in \D. 
\end{equation} 
We call the factor $\mu$ \textbf{contractivity factor} and the integer $L$ \textbf{\rev{contractivity level}}.
\end{definition}
For $k>L$, a repeated application of the contractivity relation of Definition~\ref{def:contractivity} yields,
\begin{eqnarray} \label{eqn:repeated_contractivity}
     \delta (S^k(\PP)) \le C_{\scalebox{0.7}{$\PP$}} \, \scalebox{1.2}{$\mu$}^{\floor{\frac{k}{L}}} ,
\end{eqnarray}
where $C_\PP = C_\PP(S) = \max_{\ell=0,1,\ldots,L-1}\{ \delta (S^\ell(\PP)) \} $. 

It is important to recall that any linear subdivision scheme is converging if and only if it satisfies~\eqref{eqn:contractivity_factor} with some \rev{contractivity level} $L$. This is true since the convergence of a linear subdivision scheme is equivalent to the contractivity of its associated difference scheme, see, e.g.,~\cite[Theorem 3.2]{dyn1992subdivision_book}. 

In the following, we denote by $S^k(\PP)_j$ the $j$th element of the sequence $S^k(\PP)$. 
\begin{definition}[Displacement-safe]  \label{def:DS}
A subdivision scheme $S$ is \rev{displacement-safe} in $ \D \subset \Omega^\Z$ if there exists a nonnegative constant $C_S$ such that
\begin{equation} \label{eqn:DS}
d(S(\PP)_{2j},p_j )  \le C_S \delta(\PP ) ,\quad j \in \mathbb{Z} .
\end{equation}
for any $\PP \in \D$ and where $C_S$ is a constant depending on $S$ but is independent of $\PP$. 
\end{definition}
Note that when $C_S=0$ for a converging scheme $S$, the initial points, as well as any refined sequence of points, lie on the limit curve. Thus, $S$ is termed an {\bf interpolatory scheme}. Also note that any linear subdivision scheme is \rev{displacement-safe}, see~\cite[Remark 3.3]{dyn2017manifold}.

\rev{The following theorem for CMS-valued subdivision schemes} is analogous to the convergence result of manifold-valued subdivision schemes \cite[Theorem 3.6]{dyn2017manifold} for metric-space-valued subdivision schemes.
\begin{thm}[Convergence of contractive, \rev{displacement-safe} scheme] 
\label{thm:convergence}
Let $S$ be a subdivision scheme, and let $\D\subset \Omega^\Z$, such that if $\PP \in \D$, then $S(\PP) \in \D$. If $S$ is contractive in $\D$ with a contractivity factor $\mu$ and \rev{contractivity level} $L$, and is \rev{displacement-safe} in $\D$, then it converges to a continuous limit $Q^\infty(t) \in \Omega$, $t \in \R $, from any initial data $Q^0 = (T^0,\PP^0)$ with $\PP^0 \in \D$. 
\end{thm}
\begin{proof}
To prove the convergence of $S$, it is sufficient to show that the sequence of $\Omega$-valued functions $\big\lbrace f^k(t)= \pa{A_w,Q^k; t} \big\rbrace_{k \in \N}$ is a Cauchy sequence in the metric~\eqref{eqn:metric-functions}. \rev{Here, \(Q^k=(T^k,S^k(P))\), for \(k\in\mathbb N\), is the \(k\)-th refinement of the initial data, where the parameter sequence \(T^k\) is given by the chosen primal or dual parameterization}. Let $t \in [t_{2j}^{k+1}, t_{2j+2}^{k+1})$, for some fixed $j \in \mathbb{Z}$. By the triangle inequality, we have
\begin{eqnarray} \label{eqn:proof_main_ineq}
 d \left( \pa{A_w,Q^{k+1}; t},\pa{A_w,Q^k ; t} \right)
&\le & 
d \left( \pa{A_w,Q^{k+1}; t}, S^{k+1}(\PP)_{2j} \right) + \\
& & d\left( S^{k+1}(\PP)_{2j},S^{k}(\PP)_{j} \right) + 
d\left( S^{k}(\PP)_{j} , \pa{A_w,Q^{k}; t} \right) . \nonumber
\end{eqnarray}
We now bound each of the three terms on the right-hand side of the above inequality. For the first term in the bound in~\eqref{eqn:proof_main_ineq}, using the triangle inequality again,
we get 
\begin{multline*} 
    d \left( \pa{A_w,Q^{k+1}; t}, S^{k+1}(\PP)_{2j} \right) 
\le  \\ 
 d \left( \pa{A_w,Q^{k+1}; t},  S^{k+1}(\PP)_{2j+1} \right) +
 d \left(S^{k+1}(\PP)_{2j+1}, S^{k+1}(\PP)_{2j}\right).
\end{multline*}
By the boundedness property of Definition~\ref{def:average}, the first term in the above RHS is bounded by $\delta(S^{k+1}(\PP) )$, which also bounds the second term there. The repeated contractivity~\eqref{eqn:repeated_contractivity} of $S$ finally yields  
    \begin{equation} \label{eqn:bound_term1}
    d \left( \pa{A_w,Q^{k+1}; t}, S^{k+1}(\PP)_{2j} \right) 
\le 2 \delta(S^{k+1}(\PP) ) \le  2 C_{\PP}(S) \mu^{\floor{(k+1)/L}}
\end{equation}

To bound the second term in~\eqref{eqn:proof_main_ineq}, we note that since $S$ is \rev{displacement-safe}, 
\begin{equation}\label{eqn:bound_term2}
 d\left( S^{k+1}(\PP)_{2j},S^{k}(\PP)_{j} \right)\le C_S \delta(S^{k}(\PP)) \le C_S C_{\PP}(S) \mu^{\floor{k/L}}. 
 \end{equation}
For bounding the third term in~\eqref{eqn:proof_main_ineq}, we have,  
\[ t \in [t_{2j}^{k+1}\ , t_{2j+2}^{k+1}) \subset  [t_{j}^{k}, t^k_{j+1} ) , \]
which allows us to apply the boundedness property of Definition~\ref{def:average} once more to get
\begin{equation} \label{eqn:bound_term3}
    d \left( S^{k}(\PP)_{j} , \pa{A_w,Q^{k}; t} \right) \le  \delta(S^{k}(\PP) ) \le  C_{\PP}(S) \mu^{\floor{k/L}} .
\end{equation}
Next, replacing the three terms in the bound in~\eqref{eqn:proof_main_ineq} by their bounds:  \eqref{eqn:bound_term1}, \eqref{eqn:bound_term2}  and \eqref{eqn:bound_term3}, we arrive at
\begin{align*}
d \left( \pa{A_w,Q^{k+1}; t},\pa{A_w,Q^k ; t} \right)
& \le 2 C_{\PP}(S) \mu^{\floor{(k+1)/L}} + C_{\PP}(S) (C_S  +1 )\mu^{\floor{k/L}}  \\
& \le C_{\PP}(S) (C_S +3 )\mu^{\floor{k/L}} 
\end{align*} 
The last inequality is obtained since $\mu<1$ and $\floor{(k+1)/L} \ge \floor{k/L} $. 

Setting $B = C_{\PP}(S) (C_S +3) \mu^{-1}$ which is independent of $k$, and noting that $\floor{k/L}
\ge\frac{k}{L}-1 $, we conclude that
\[ d^k(t) = d\left( \pa{A_w,Q^{k+1}; t},\pa{A_w,Q^k ; t} \right)
\le B (\mu^{1/L})^k . \]
Finally, since $\mu^{\frac{1}{L}}<1$, the sequence $\{ d^k(t) \}_{k\in\Z} $ is summable and therefore, for any $m \in \N$,
\[
d\left( \pa{A_w,Q^{k+m}; t},\pa{A_w,Q^k ; t} \right)
\le \sum_{\ell=k}^{k+m} d^\ell(t) \le \sum_{\ell=k}^{\infty}  B (\mu^{1/L})^\ell \le   \frac{B}{1-\mu^{1/L}} (\mu^{1/L})^k .
\]
Thus, $\bigg\lbrace \pa{A_w,Q^k ; t} \bigg\rbrace_{k\in\Z}$ is a Cauchy sequence in the metric~\eqref{eqn:metric-functions}. 

By the continuity of $\pa{A_w,Q^k ; t} $ (see the text below Definition~\ref{def:piecewise_average}) and since the convergence is uniform in $t$ and the space is complete, we get a continuous limit in $\Omega$.
\end{proof}

%----------------------------------------------------------
\subsection{Application of the convergence theorem to elementary schemes with intrinsic averages} \label{subsec:app_convergence}

As a first application of Theorem~\ref{thm:convergence}, we prove the convergence of elementary schemes that are based on intrinsic averages. This generalizes similar results on manifold-valued subdivision schemes that use geodesic averages~\cite{dyn2017global, dyn2017manifold}.

We start by recalling the notion of \bsq{elementary scheme} defined in~\eqref{eqn:elementary_scheme}. A scheme $S$ in a CMS with average $A_\omega$ is called \bsq{elementary} if 
\begin{equation} \label{eqn:recall1}
S(\PP)_{2j}=p_j,\quad {\text and}\quad S(\PP)_{2j+1}= A_{\frac{1}{2}}(p_j , p_{j+1}).
\end{equation}
We also recall the notion of \bsq{intrinsic average}. The average $A_w$ is called \bsq{intrinsic} if it has the metric property. Namely, for all $x,y\in \Omega$
\begin{equation}\label{eqn:recall2}
 d(x,A_{\omega}(x,y)) = \omega d(x,y),\, \quad \omega \in [0,1].
 \end{equation}

\begin{prop} \label{prop:convergence_of_intrinsic_elementary_scheme}
    Any elementary subdivision scheme with an intrinsic average is \rev{strongly} converging.
\end{prop}
\begin{proof}
It is enough to show that with an intrinsic average, an elementary scheme $S$ satisfies the conditions of Theorem~\ref{thm:convergence}. First, we note that $S$ in view of \eqref{eqn:recall1} is defined in $\Omega^\Z$,  since the average $A_\omega$ is defined for all $x,y\in\Omega$. Thus $\D$ in Theorem~\ref{thm:convergence} is $\Omega^\Z$, and the requirement on $S$ that for any $\PP \in \D$ also $S(\PP)\in \D$ is trivial. Next, by the first equality in~\eqref{eqn:recall1}, the scheme is interpolatory and thus has the \rev{displacement-safe} property. Finally by the second equality in~\eqref{eqn:recall1} and by the metric property \eqref{eqn:recall2} of the intrinsic average,   
we get
\begin{equation} \label{eqn:contractivity_elementary}
    \delta \big(S(\PP) \big)\le \frac{1}{2} \delta (\PP).
\end{equation}
Thus, $S$ has a contractivity factor $ \frac{1}{2}$ and a \rev{contractivity level} $1$,
and we conclude that $S$ satisfies the conditions of Theorem~\ref{thm:convergence}.
\end{proof}
   
\begin{remark}
Two comments follow, the first is about the locality of the limit, and the second is about the convergence of more schemes than the elementary.
\begin{enumerate}
\item 
\rev{
Since elementary schemes are interpolatory and use only neighboring data, their limits have a particularly simple finite-dependence property: for \(t\in[t_i^0,t_{i+1}^0]\), the limit depends only on the two initial points \(p_i^0\) and \(p_{i+1}^0\). This sharper observation is specific to elementary schemes and is stronger than the general finite-dependence bound of Remark~\ref{rem:locality}.
}
\item 
Similar to the manifold setting, Proposition~\ref{prop:convergence_of_intrinsic_elementary_scheme} can be extended to a wide family of schemes using the global approach in~\cite{dyn2017global}. An example of such a scheme is provided here in Subsection~\ref{subsec:application_second_prox}.
\end{enumerate}
\end{remark}

Two examples of elementary schemes in metric spaces with intrinsic average are Example~\ref{example:compact sets} and Example~\ref{example:Wasserstein}. Both examples yield converging subdivision schemes, as guaranteed by Proposition~\ref{prop:convergence_of_intrinsic_elementary_scheme}. In the first case, the subdivision scheme generates a converging sequence of sequences of compact sets, which tend to a continuous curve of compact sets see~\cite{dyn2014approximation}. In the second case, the scheme outputs a continuous curve in the Wasserstein space~\cite{banerjee2025efficient}.

%----------------------------------------------------------
\subsection{Approximation property of elementary schemes with intrinsic averages}

We conclude this section with an approximation result showing that the limit of an elementary scheme with an intrinsic average approximates a Lipschitz-continuous curve. A continuous curve $\Gamma \colon \R \to \Omega$ is Lipschitz continuous if there exists a constant $L_{\Gamma}>0$ such that
\begin{eqnarray}
    d(\Gamma(x),\Gamma(y)) \le L_{\Gamma} \abs{x-y}, \quad x,y \in \R .
\end{eqnarray}
\begin{thm}
    Let $(\Omega,d)$ be a CMS, let $\Gamma \colon \R \to \Omega$ be a Lipschitz continuous curve with Lipschitz constant $L_\Gamma$, let $\PP^0_j = \Gamma(t_j^0)$, $j\in \Z$, let  $h = \sup_{i\in\Z}\abs{t^0_{i+1}-t^0_i} $. If $S$ is an elementary scheme with an intrinsic average defined on  $\Omega$, then
    \[ d(\Gamma(t),S^\infty(\PP^0)(t)) \le C h, \quad  t \in \R .\]
The constant $C$ is independent of $t$ and $h$.
\end{thm}
\begin{proof}
Let $ I_k = \{1,\ldots,  2^k-1  \}$, $k\in \N$. We begin with proving by induction the following auxiliary claim,
\begin{equation} \label{eqn:app_induction_claim}
 d(p_{2^k j+(r-1)}^k,p_{2^k j+r}^k) \le 2^{-k} L_{\Gamma }h , \quad r \in I_k , \quad k \in \N .
\end{equation}
For the base case $k=1$, by the definition of the elementary scheme, followed by the metric property of the intrinsic average and the Lipschitz continuity of $\Gamma$, we have
\[ d(p^1_{2j},p^1_{2j+1}) =  d(p^0_j,p^1_{2j+1})= \frac{1}{2} d(p^0_{j},p^0_{j+1}) \le  \frac{1}{2} L_{\Gamma } h . \]
Assume that~\eqref{eqn:app_induction_claim} holds for $k=\ell$, we show it for $k=\ell+1$. 
First consider odd $r$. Then, 
\begin{equation} \label{eqn:r_odd}
d(p_{2^{\ell+1} j+r-1}^{\ell+1},p_{2^{\ell+1} j+r}^{\ell+1}) = d(p_{2^{\ell} j+\frac{r-1}{2}}^{\ell},p_{2^{\ell+1} j+r}^{\ell+1}) . 
\end{equation}
However, in the elementary scheme, $p_{2^{\ell+1} j+r}^{\ell+1}$ is the average between $p_{2^{\ell} j+\frac{r-1}{2}}^{\ell}$ and $p_{2^{\ell} j+\frac{r-1}{2}+1}^{\ell}$. Therefore, by the induction hypothesis and by the metric property of the intrinsic average, we get that
\[ d(p_{2^{\ell+1} j+r-1}^{\ell+1},p_{2^{\ell+1} j+r}^{\ell+1}) = \frac{1}{2} d(p_{2^{\ell} j+\frac{r-1}{2}}^{\ell}, p_{2^{\ell} j+\frac{r-1}{2}+1}^{\ell}) \le \frac{1}{2}\left( 2^{-\ell} L_{\Gamma }h  \right)  = 2^{-(\ell+1)} L_{\Gamma }h .  \]
In case $r$ is even, instead of~\eqref{eqn:r_odd} we consider
\[ d(p_{2^{\ell+1} j+r}^{\ell+1},p_{2^{\ell+1} j+r+1}^{\ell+1}) = d(p_{2^{\ell} j+\frac{r}{2}}^{\ell},p_{2^{\ell+1} j+r+1}^{\ell+1}) . \]
The rest of the proof of this case is analogous to the case where $r$ is odd. This concludes the proof of the induction. 

A conclusion from~\eqref{eqn:app_induction_claim} is that 
\[ d(p_{2^k j}^k,p_{2^k j+r}^k) \le  \sum_{v=1}^{r} d(p_{2^{\ell+1} j+v-1}^{\ell+1},p_{2^{\ell+1} j+v}^{\ell+1}) \le r 2^{-k} L_{\Gamma }h \le L_{\Gamma }h , \quad  r \in I_k .\]
Note that the right-hand side (RHS) above is independent of $k$. In addition, since $p^0_j=p_{2^k j}^k$, we have
\[ d(p_{j}^0,p_{2^k j+r}^k) \le L_{\Gamma }h . \]
Taking the limit with respect to $k$ yields, for any $t \in [t^0_i,t^0_{i+1}]$, that
\begin{eqnarray*}
    d(\Gamma(t),S^\infty(\PP)(t)) &\le&  d(\Gamma(t),\Gamma(t^0_i)) + d(\Gamma(t^0_i),S^\infty(\PP)(t))   \\
    &\le&  L_{\Gamma}(t-t_i^0) + L_{\Gamma}h \le 2L_{\Gamma}\, h \, .
\end{eqnarray*}
\end{proof}

We conclude this section by observing two types of convergence.
\begin{definition}
     A subdivision scheme is called \textbf{strongly} converging scheme in $\D \subset \Omega^\Z$ if it converges from any initial sequence there. On the other hand, a scheme is termed \textbf{weakly} converging in $\D$ if it converges to a continuous limit only from initial data in $\D$, satisfying a condition on the size of $\delta (\PP)$. 
\end{definition}

Recall that any converging linear scheme must obey the contractivity of Definition~\ref{def:contractivity}, and it also satisfies the \rev{displacement-safe} property of Definition~\ref{def:DS} (see the discussions below the definitions). Therefore, we conclude that any converging linear scheme is strongly converging, while weak converging schemes in linear spaces are nonlinear schemes that are in \textbf{proximity} to linear converging schemes, as done, for example, in~\cite{wallner2005convergence}. \rev{Strong convergence results also appear in nonlinear settings, for nonpositive curvature space, see~\cite{ebner2014stochastic}}. In the next section, we extend the analysis of convergence by proximity to CMS-valued subdivision schemes.

%% =============================================
\section{Proximity of the first type} \label{sec:proximity1}

Theorem~\ref{thm:convergence} demonstrates how convergence can be established when the scheme is contractive. However, proving contractivity is often challenging. In classical non-stationary linear subdivision schemes, where the refinement rule varies with the level, an equivalent condition was introduced, showing that if a non-stationary process is sufficiently close to a stationary one, it inherits the regularity of the stationary scheme, see e.g.,~\cite{dyn2003exponentials}. This idea was later extended to the concept of proximity, where nonlinear subdivision schemes were compared to linear ones to prove convergence and smoothness of the nonlinear scheme~\cite{wallner2005convergence}.
Here, we further develop this approach by allowing proximity between \textbf{two nonlinear schemes}. We then show, within the CMS framework, that if one of the two schemes is contractive, weak convergence is inherited by the other scheme.

%----------------------------------------------------------
\subsection{Extending classical proximity to complete metric spaces}

We start with our first proximity condition:
\begin{definition}[Proximity of the first type]
\label{def:proximity}
   Let $X=(\Omega,d)$ be a complete metric space and let two refinement operators $S_1$ and $S_2$ be defined in $\Omega$. The operators $S_1$ and $S_2$ are in proximity in $\D \subset \Omega^\Z$, if there exist $C>0$ and $\varepsilon>0$ such that for any $\PP \in \D$,
        \begin{equation} \label{eqn:proximity}
            \sup_{j\in \Z} d(S_1(\PP)_j,S_2(\PP)_j) \le C (\delta (\PP))^{1+\varepsilon}. 
        \end{equation}
        Here $C$ and $\varepsilon$ are independent of $\PP$. 
\end{definition}
The above condition is a generalization of the classical proximity condition, where $S_1$ is a linear
subdivision scheme in a Euclidean space, and $S_2$ is a subdivision scheme with values in a manifold contained in the same Euclidean space, see~\cite[Section 3.3]{wallner2005convergence}. The first proximity condition was then extended from  $\varepsilon=1$ to $\varepsilon>0$ in~\cite[Definition 3]{wallner2006smoothness}, which is a \rev{slightly weaker condition than~\eqref{eqn:proximity} for small $\delta(\PP)$.} 
Here, we further extend the setting to CMS, without assuming that one of the refinement operators is linear. \rev{In other words, this condition treats the two schemes $S_1$ and $S_2$ symmetrically. In the following theorem, one scheme is assumed to be strongly convergent ($S$); proximity then implies weak convergence of the other scheme ($W$).}

\begin{thm} \label{thm:main_proximity}
    Let $S$ be a subdivision scheme defined on $\Omega$ relative to parameter sequences $\{T^k\}_{k=0,1,2,\ldots}$ (see Definition~\ref{def:subdivision}), and such that there exists $\D \subset \Omega^\Z$ where $S$ is contractive  with contractivity factor $\mu$ and \rev{contractivity level} $L$. Let $W$ be a subdivision scheme defined on $\Omega$ satisfying:
    \begin{enumerate}
        \item % (parameterization)
        $W$ is defined relative to the same parameter sequences as $S$, namely  relative to $\{T^k\}_{k=0,1,2,\ldots}$.
        \item % (closeness)
        For all $\PP \in \D,\   W(\PP) \in \D$.
        \item %(Proximity) 
        The subdivision schemes $S^L$ and $W^L$ are in proximity of Definition~\ref{def:proximity} in $\D$, with a positive constant $C=C_{(S^L,W^L)}$ and $\varepsilon>0$ in~\eqref{eqn:proximity}.
        \item %(\rev{displacement-safe})
        $W$ is \rev{displacement-safe} in $\D$.
    \end{enumerate}
    Then, $W$ converges from any $\PP \in \D$ satisfying
    \begin{equation} \label{eqn:boundondeltaP}
    \delta (\PP) < \left(\frac{1-\mu }{2C_{(S^L,W^L)}} \right)^{1/\varepsilon}. 
    \end{equation}
\end{thm}
The proof of the Theorem follows the lines of the proof of~\cite[Theorem 2]{wallner2005convergence} and generalizes it for CMS, nonlinear $S$, arbitrary \rev{contractivity level} $L$, and $\varepsilon < 1$. We provide the proof here for completeness.
\begin{proof}    
First, we show the existence of a contractivity factor of $W$ for any $\PP \in \D$. Indeed, by the triangle inequality, for any $\ell \ge L$, 
\begin{multline} \label{eqn:proximity_bound1}
    d((W^\ell \PP)_i,(W^\ell \PP)_{i+1})
  \le   
  d\left((W^L W^{\ell-L} \PP)_i,(S^L W^{\ell-L} \PP)_{i} \right) + \\
  d\left((S^L W^{\ell-L} \PP)_{i},(S^L W^{\ell-L} \PP)_{i+1}\right) + d\left((S^L W^{\ell-L} \PP)_{i+1},(W^L W^{\ell-L} \PP)_{i+1}\right)
\end{multline}
Since $S$ has a contractivity factor $\mu $ and a \rev{contractivity level} $L$, and since $W$ maps $\D$ into itself, the middle term in the RHS of~\eqref{eqn:proximity_bound1} is bounded by $\mu \, \delta \left(W^{\ell-L}(\PP)\right)$. For the other two terms in~\eqref{eqn:proximity_bound1}, we use proximity and bound each term by $C_{(S^L,W^L)}\left(\delta (W^{\ell-L}(\PP)) \right)^{1+\varepsilon}$. Replacing the three terms in~\eqref{eqn:proximity_bound1} by their bounds, and taking the supremum over $i$, we get
\begin{equation} \label{eqn:recursion}
    \delta (W^\ell (\PP)) 
  \le \mu \, \delta (W^{\ell-L}(\PP))+ 2C_{(S^L,W^L)}\left(\delta (W^{\ell-L}(\PP)) \right)^{1+\varepsilon},
\end{equation}
For $\nu$ satisfying $0<\nu < \left(\frac{1-\mu }{2C_{(S^L,W^L)}} \right)^{1/\varepsilon}$, we define
\begin{equation} \label{eqn:mu_W_def}
 \muW : = \mu + 2C_{(S^L,W^L)} \nu^{\varepsilon}  .  
\end{equation} 
Clearly, $ \muW \in (0,1)$. To complete the proof, we claim the following,
 \begin{equation} \label{eqn:claimforinduction}
      \text{If  } \delta (\PP) \le \nu \text{   then  } \delta (W^{k L} (\PP)) \le \muW^k \delta (\PP) , \quad k \in \mathbb{N}. 
 \end{equation}
We use~\eqref{eqn:recursion} and show~\eqref{eqn:claimforinduction} by induction on $k$. For $k=1$, using $\ell=L$ in \eqref{eqn:recursion} we get,
\begin{equation} \label{eqn:proximity_first_step}
\begin{aligned}
\delta (W^L (\PP)) & \le \mu \delta (\PP) + 2C_{(S^L,W^L)} (\delta (\PP))^{1+\varepsilon} \\ & \le \left(\mu + 2C_{(S^L,W^L)}(\delta (\PP))^\varepsilon \right) \delta (\PP) \\  & \le \left(\mu + 2C_{(S^L,W^L)}\nu^\varepsilon \right) \delta (\PP) = \muW \delta (\PP) .
\end{aligned}  
\end{equation}

 The last inequality follows from the monotonicity of $x^\varepsilon$. Thus, \eqref{eqn:claimforinduction} holds for $k=1$. Assume~\eqref{eqn:claimforinduction} holds for $k$; we now show that it holds for $k+1$. Note that by the above induction assumption and since $(\muW)^k \in (0,1)$, we get from~\eqref{eqn:claimforinduction} that $\delta (W^{k L} (\PP)) \le \delta (\PP) \le \nu$. Therefore, by~\eqref{eqn:recursion} with $\ell = (k+1)L$ and the inductive assumption we obtain,
 \begin{equation*}
\begin{aligned}
    \delta (W^{(k+1)L} (\PP)) & \le \mu  \delta (W^{k L} (\PP)) +  2C_{(S^L,W^L)} \left(\delta (W^{k L} (\PP))\right)^{1+\varepsilon} \\
    & = \left(\mu + 2C_{(S^L,W^L)} (\delta (W^{k L}(\PP))^\varepsilon \right)  \left(\delta (W^{k L} (\PP))\right) \\ &\le \muW  \left(\delta (W^{k L} (\PP))\right) \le \muW^{k+1} \delta (\PP) .
\end{aligned}  
\end{equation*}
The conclusion is that under the condition $\delta (\PP) \le \nu$, the subdivision scheme $W$ has a contractivity factor $\muW$ with a \rev{contractivity level} $L$. Note that $\nu$ is independent of $\PP$. 

Finally, the weak convergence of the scheme $W$ follows by Theorem~\ref{thm:convergence} since $W$ has a contractivity factor when $\delta(\PP)$ is small enough, it is \rev{displacement-safe} in $\D$, and it maps $\D$ onto itself.
\end{proof}

\begin{remark}
Two comments about the proximity of Theorem~\ref{thm:main_proximity}:
\begin{enumerate}
    \item The condition~\eqref{eqn:boundondeltaP} implies that $W$ is converging weakly in $\D$.
    \item The contractivity factor $\muW$ in~\eqref{eqn:mu_W_def} can be made arbitrarily close to $\mu$ of $S$, by choosing small enough $\nu$.
    \item In view of Theorem~\ref{thm:main_proximity}, the proximity condition is of interest in cases $\delta (\PP)< 1$. In these cases, when $\varepsilon$ is smaller, the proximity required between $S$ and $W$ in~\eqref{eqn:proximity} is weaker and the condition over the data density $\delta (\PP)$ in~\eqref{eqn:boundondeltaP} is more restrictive when $\frac{1-\mu }{2C_{(S^L,W^L)}}<1$, which is more probable. 
\end{enumerate}
\end{remark}

%----------------------------------------------------------
\subsection{Application of the proximity of the first type to two Hermite elementary schemes} \label{subsection:application_first_type_proximity}

In the classical Hermite interpolation of functions, we are provided with samples of a smooth function along with its derivatives at the sampling points. Specifically, a first-order Hermite sample consists of the function's value and its first derivative at a sampling location. In geometric terms, this corresponds to sampling a smooth curve, where a first-order sample consists of a point on the curve and the tangent vector of the curve at that point.

Each tangent vector comprises two distinct components: direction and magnitude (length). The magnitude of a tangent vector represents the traversal velocity of the curve at the sampling point. If we assume the curve is regular, namely that no zero-length tangent vectors occur, we can always employ an arc-length parameterization. Under such parameterization, we traverse the curve at a constant, unit velocity, thereby standardizing all tangent magnitudes to one.

Consequently, when focusing purely on geometric properties, the tangent vectors can be treated as unit vectors, simplifying the representation so that the relevant geometric information is entirely captured by their directions. This simplification is formalized in the following definition. 
\begin{definition}[The first-order geometric Hermite interpolation problem for curves in $\R^n$]
    Consider a regular $C^1$ curve $\gamma \colon \R \to \R^n$. We aim to construct an approximating curve $ \widetilde{\gamma} \approx \gamma$ from a set of discrete samples of $\gamma$, at $\{t_j:j=1,2,\ldots,J\} $. Each sample consists of the pair of n-tuples $(\gamma(t_j),\dot\gamma(t_j))$, where $\dot\gamma(t_j)$ is the normalized tangent vector of $\gamma$ at $t_j$. The interpolation conditions read 
    \begin{equation} \label{eqn:int_condition_hermite}
        \widetilde{\gamma}(t_j)=\gamma(t_j)\quad \text{and}\quad  
 \dot{\widetilde{\gamma}}(t_j)=  \dot\gamma(t_j), \quad j=1,2,\ldots,J.
    \end{equation}
\end{definition}

The example we describe here directly links to Example~\ref{example:hermiteaverage}, where we consider the domain \(\Omega \subset \mathbb{R}^n \times \mathbb{S}^{n-1}\). The data $\PP = \{(p_i,v_i)\}_{i \in \Z}$ consists of a sequence of pairs, each pair holds the point and its associated normalized tangent. We measure the distance between two pairs $(p,v)$ and $(q,w)$ with the $\ell_2$ mixed metrics,
\begin{equation} \label{eqn:metric_hermite}
        d\big((p,v),(q,w)\big) = \sqrt{\norm{p-q}^2 + (d_S(v,w))^2} ,
\end{equation}
where $d_S(\cdot,\cdot)$ is the geodesic distance on the unit sphere in $\R^n$ (the length of a connecting great circle arc), and $\norm{\cdot-\cdot}$ is any norm in $\R^n$.

%----------------------------------------------------------
\subsubsection{Two elementary Hermite schemes}

We begin with the first naive elementary scheme in the above metric space $X=(\Omega,d)$. The corresponding average $A_w$ of any two pairs $(p,v),(q,u) \in \Omega$, provided $v$ and $u$ are not antipodal points, is
\begin{equation*}
    A_w\left( (p,v),(q,u) \right)  = \left( (1-w)p+wq, s_w(v,u) \right) .
\end{equation*}
Here, $s_w(v,u) = \frac{\sin((1 - w)\theta)}{\sin(\theta)} v + \frac{\sin(w\theta)}{\sin(\theta)} u$, where $\theta = \arccos(v \cdot u)$. This average is intrinsic with respect to the metric~\eqref{eqn:metric_hermite} and the structure of the tensor product $\Omega$. The induced elementary scheme takes the form~\eqref{eqn:elementary_scheme} and its strong convergence is guaranteed by Proposition~\ref{prop:convergence_of_intrinsic_elementary_scheme}. The explicit refinement rules of this first elementary scheme are
\begin{equation} \label{eqn:naive_hermite_elemntary}
    S(\PP)_{2i} = \left( p_i, v_i \right) , \quad
 S(\PP)_{2i+1} = \left( \frac{1}{2}(p_i + p_{i+1}) , s_{\frac{1}{2}}(v_i,v_{i+1}) \right) , \quad i \in \Z.
\end{equation}
The elementary geometric scheme in~\eqref{eqn:naive_hermite_elemntary} refines the points and tangents separately, making it impossible to improve the accuracy of the limit curve to the sampled curve due to the additional information of the tangents. This observation suggests that the elementary scheme should integrate both position and tangent information effectively. The following scheme realizes this conclusion. It is a circle-preserving geometric scheme inspired by~\cite{hofit2023geometric}, see Example~\ref{example:hermiteaverage}. This scheme uses, for the refinement of the points, a cubic Hermite Bézier curve that meets the interpolation conditions defined in~\eqref{eqn:int_condition_hermite}. Its explicit refinement rules are 
\begin{equation} \label{eqn:hofit_like_hermite_scheme}
W(\PP)_{2i} = \left( p_i, v_i \right) ,  \quad 
W(\PP)_{2i+1}=\left(\frac{1}{2}(p_i+p_{i+1})-\frac{\alpha_i}{8}(v_{i+1}-v_i), s_{\frac{1}{2}}(v_i,v_{i+1}) \right) , \quad i \in \Z.
\end{equation}   
Here, $\alpha_i = \frac{\norm{p_{i+1}-p_{i}}}{\cos^2{(\theta_i/4)}}$ with
$\theta_i = d_S\left(v_{i},v_{i+1} \right)$. Note that this is analogous to the construction in Example~\ref{example:hermiteaverage} where $c=\frac{\alpha_i}{3}$.

%----------------------------------------------------------
\subsubsection{The proximity and its implications}

The first result shows the proximity between the two elementary Hermite schemes:
    \begin{lemma} \label{lemma:proximity_Hermite}
         Let $\PP \in \D$ be Hermite data that satisfies $\norm{p_{j+1}-p_j}<1$ and $\theta_j = d_S \left(v_j,v_{j+1} \right) < 1$, $j\in\Z$. Then, there exists $C>0$, independent of $\PP$, such that,
        \begin{equation} \label{eqn:proximity_hermite}
            \sup_{j\in \Z} d\big(S(\PP)_j,W(\PP)_j \big) \le C \big(\delta (\PP) \big)^2 . 
        \end{equation}
    \end{lemma}
    \begin{proof}
For the even indices, by interpolation, we have $d\left( S(\PP)_{2i},W(\PP)_{2i} \right)=0$. For the odd indices we have,
\begin{equation*}
  d\left( S(\PP)_{2i+1},(W(\PP)_{2i+1} \right)^2 = \norm{\frac{ \alpha_i }{8}(v_{i+1}-v_{i})}^2  . 
\end{equation*}
Since $\theta_i<1$ it is bounded away from $2\pi$, and we have that ${\cos^2\left({\theta_i}/{4}\right)} \ge {\cos^2\left({1}/{4}\right)}$. Also, note that $\norm{v_{i+1}-v_{i}}\le d_S(v_{i+1},v_i)$, and that any separate component of the metric~\eqref{eqn:metric_hermite} is bounded by the whole metric. These three observations lead to,  
\begin{eqnarray}\label{eqn:BoundOnTermWithTangents}
         \norm{\frac{ \alpha_i }{8}(v_{i+1}-v_{i})} &=& \norm{\frac{\norm{p_{i+1}-p_i} }{8 \cos^2(\theta_i/4)}(v_{i+1}-v_{i})} \\ 
        &\le& d_S\big((v_{i+1},v_{i}) \big) \norm{p_{i+1}-p_i} \frac{1}{8\cos^2(1/4)} \le \frac{1}{8\cos^2(1/4)}  \delta(\PP)^2 .
\end{eqnarray}
 This completes the proof of the lemma with $C= \frac{1}{8\cos^2(1/4)} $.   
\end{proof}

To apply Theorem~\ref{thm:main_proximity}, we note that $S$ is an elementary scheme with an intrinsic average and thus it strongly converges with contractivity factor $\frac{1}{2}$ and \rev{contractivity level} $1$. Thus, we have four conditions on $W$ to check. First, the schemes $W$ and $S$ have the same parameterization, Second, by Lemma~\ref{lemma:proximity_Hermite}, $W$ and $S$ are in proximity of the first type, and third, since both schemes are interpolatory, we automatically obtain the \rev{displacement-safe} condition. Therefore, the only condition left to check is stated and proved in the next lemma. 
\begin{lemma} \label{lemma:hermite_closeness}
    Let $\D$ consist of all Hermite data $\PP=\{(p_i,v_i)\}_{i\in\Z}$, satisfying 
    $\theta = \sup_ {i\in\Z} d_S(v_i,v_{i+1}) < 1 $,
    and $\sup_{i\in\Z}\ \norm{p_{i+1} - p_i}< 1$. 
    Then, for any $\PP \in \D,\   W(\PP) \in \D$.
\end{lemma}
\begin{proof}
Denote $W(\PP) = \widetilde\PP=\{ (\widetilde{p}_{i},\widetilde{v}_{i})\}_{i\in \Z}$, 
and $\widetilde{\theta} =\sup_{i\in\Z} d_S(\widetilde{v}_i, \widetilde{v}_{i+1})$.
By the metric property of the average of the tangents, the angular distance between consecutive tangents does not increase after applying the refinement rules of $W$ for the tangents (see
equation~\eqref{eqn:hofit_like_hermite_scheme}), and therefore, 
 $\widetilde{\theta} \le\theta< 1$.
Thus $\widetilde {\PP}$ satisfies the condition on the angular angles between consecutive tangents, and it remains to show that $\norm{\widetilde{p}_{i+1} - \widetilde{p}_i}<1 $.

By the definition of the metric (see equation~\eqref{eqn:metric_hermite}) and since $W$ is an interpolatory scheme,
$$d\left(W(\PP)_{2i\pm1},W(\PP)_{2i}\right)^2=\norm{\widetilde{p}_{2i\pm1}-p_{i}}^2+d_S^2\left(\widetilde{v}_{2i\pm1},v_i\right).$$
Now we bound the first term in the RHS of the above equation. By the refinement rules of $W$ for the points (see  equation~\eqref{eqn:hofit_like_hermite_scheme}) 
$$\norm{\widetilde{p}_{2i\pm1}-p_i}=\norm{\pm\frac{1}{2}(p_{i\pm1}-p_i)\pm \frac{\alpha_i}{8}(v_{i\pm1}-v_i)}\le\frac{1}{2}\norm{p_{i\pm1}- p_i}+ \frac{\alpha_i}{8}\norm{v_{i\pm1}-v_i}.$$
Since $\alpha_i = \frac{\norm{p_{i+1}-p_{i}}}{\cos^2{(\theta_i/4)}}$, and $\theta_i< 1$,
$\cos{(\theta_i/4)} > \cos(\frac{1}{4}) > \cos(\frac{\pi}{3})=\frac{1}{2}$, and we obtain,
$\frac {\alpha_i}{8}<\frac{1}{2}\norm{p_{i+1}-p_{i}}$. Recalling that $\norm{v_{i+1}-v_{i}}\le d_S(v_{i+1},v_i) < 1 $, we finally arrive at
$$\norm{\widetilde{p}_{2i\pm1}-p_i}<\norm {p_{i+1}-p_i}<1 ,$$ 
implying that $\widetilde \PP\in \D$.
\end {proof}

In view of Lemmas~\ref{lemma:proximity_Hermite} and~\ref{lemma:hermite_closeness}, all the conditions of Theorem~\ref{thm:main_proximity} hold on the domain $\D$. Hence,
\begin{cor}
    The scheme $W$ is weakly converging in $\D$ for any $\PP\in\D$ that satisfies $\delta(\PP) <1$.   
    \end{cor}

%% =============================================
\section{Proximity of the second type} \label{sec:proximity2}

In this section, we introduce a second family of proximity conditions. This proximity was first suggested for a specific setting in~\cite[Remark 3.7]{lipovetsky2019c1}. We formulate it for CMS-valued subdivision schemes, generalize it by adding the notion of order, and finally show how to use it for proving convergence and later for proving smoothness in $\R^n$.

\subsection{The proximity and its induced convergence theorem}
We start with the proximity condition.
\begin{definition}[Proximity of the second type of order $m$] \label{def:2nd_type_prox_def}
Let $X=(\Omega,d)$ be a complete metric space and let two subdivision schemes $S_1$ and $S_2$ be defined on $\Omega^\Z$. We say that the subdivision scheme $S_1$ is in proximity of order $m$ to $S_2$ in $\D \subset \Omega^\Z$ if there exists a constant $E>0$ and a factor $\eta$ such that for any $\PP \in \D$,
\begin{equation} \label{eqn:2nd_type_proximity}
\sup_{i \in \Z} d\left( S_1(\PP^j)_i,S_2(\PP^j)_i \right) \le E \eta^{j+1} , \quad 0 < \eta < \left( \frac{1}{2} \right)^{m} ,\quad j \in \Z_+ ,
\end{equation}
where $\PP^j = S_1^j(\PP)$, $E>0$ is independent of $j$, and $\eta$ is independent of $\PP$ and $j$.
\end{definition}
\begin{remark}
A few comments about the proximity of Definition~\ref{def:2nd_type_prox_def}: 
\begin{enumerate}
    \item The order $m$ indicates how close is the proximity between the two schemes. In particular, as $m$ grows, the two subdivision operators act increasingly similar on the sequences of refined points by $S_1$. Note that there is a natural hierarchy, which means that if $S_1$ is in proximity of the second type of order $m$ to $S_2$, it is also in proximity of any order less than $m$ to $S_2$. The meaning of the order becomes clearer later in this section.  
    \item Additional notes that highlight differences between the first type of proximity of Definition~\ref{def:proximity} and the second type proximity of Definition~\ref{def:2nd_type_prox_def}.
    \begin{enumerate}[label=(\roman*)]
        \item The first type of proximity focuses on the difference between the two schemes after one refinement step, that is, between the refinement operators. The second type follows the proximity of $ S_1$ to $S_2$ along the refinement levels of $ S_1$, which is required to become stronger as the refinement level increases.
        \item The constant $E$ of the second type proximity may depend on the sequence $\PP$ and, in particular, on $\delta(\PP)$, while the constant in~\eqref{eqn:proximity} is independent of $\PP$.
        \item 
        An alternative requirement for proximity of the second type is~\eqref{eqn:2nd_type_proximity} but only for $j \ge J$ with some fixed $J>1$, $J\in\N$, which is independent of $\PP$.
        
        This requirement is sufficient since for each sequence $\PP$ we can select $E$, so that the bound will be valid for $j=0,\ldots, J$. Since $J$ is independent of $\PP$ then $E$ depends on $\PP$.
        \item
        The proximity of the second type is not symmetric with respect to the two schemes, while the proximity of the first type is.
    \end{enumerate}
\end{enumerate}
\end{remark}
In the following, we require the proximity of $S_1^L$ to $S_2^L$. In this case, \eqref{eqn:2nd_type_proximity} becomes
\[ \sup_{i \in \Z} d\left( S_1^L( S_1^{jL}(\PP))_i,S_2^L( S_1^{jL}(\PP))_i \right) \le E \eta^{j+1} , \quad \eta < \left( \frac{1}{2} \right)^{m} ,\quad j \in \Z_+ . \]
In particular, note that the power of $\eta$ scales like $\cdot/L$ with respect to the number of refinements.

The following result shows the convergence of a subdivision scheme based on a second type proximity of order $0$ to a scheme with a contractivity factor. \rev{Here, strong convergence is transferred from one scheme to the other.} 
\begin{thm} \label{thm:convergence_proximity_2ndType}
    Let $S_2$ be a subdivision scheme defined on $\Omega^\Z$ relative to parameter  sequences $\{T^k\}_{k=0,1,2,\ldots}$ (see Definition~\ref{def:subdivision}), and such that there exists $\D \subset \Omega^\Z$ where $S_2$ has a contractivity factor $\mu\in (0,1)$ and a \rev{contractivity level} $L\in \N$. Let $S_1$ be a subdivision scheme defined on $\Omega$ satisfying:
    \begin{enumerate}
        \item % (parameterization)
        $S_1$ is defined relative to the same parameter sequences as $S_2$, namely  relative to $\{T^k\}_{k=0,1,2,\ldots}$.
        \item % (closeness)
        For all $\PP \in \D$,  $S_1(\PP) \in \D$.
        \item %(Proximity) 
        The subdivision scheme $S_1^L$ is in proximity of the second type of order $m=0$  to $S_2^L$ in $\D$.
        \item 
        $S_1$ is \rev{displacement-safe}.
    \end{enumerate}
Then, $S_1$ converges to a continuous limit curve in $\D$ from any $\PP \in \D$.
\end{thm}
\begin{proof} 
Let $Q^\ell=\{(t_j^\ell,S_1^{\ell}(\PP)_j)\}_{j\in \Z}$, with  $\ell \in \N$, the $\ell$th refinement of $\PP$ using $S_1$. We show that the sequence $\big\lbrace \pa{A_w,Q^\ell; t} \big\rbrace_{\ell \in \N}$ is a Cauchy sequence in the metric~\eqref{eqn:metric-functions}. For this, it is enough to show that the distance $d \big( \pa{A_w, Q^{\ell+1}; t},\pa{A_w, Q^\ell ; t} \big) $ is summable over $\ell$ uniformly in $t\in \R$.  

\rev{
Let \(\ell=mL+r\), where \(m=\lfloor \ell/L\rfloor\) and \(r\in\{0,\ldots,L-1\}\). Since \(S_1(\D)\subset \D\), we have \(S_1^r(\PP)\in \D\). Applying the second-type proximity of \(S_1^L\) to \(S_2^L\) to the initial sequence \(S_1^r(\PP)\), there exists a constant \(E_r>0\), independent of the refinement level, such that
\[ \sup_{j\in\mathbb Z}
d\!\left(
S_1^L\bigl((S_1^L)^q(S_1^r(\PP))\bigr)_j,
S_2^L\bigl((S_1^L)^q(S_1^r(\PP))\bigr)_j
\right) \le E_r\eta^{q+1}, \qquad q\in\mathbb Z_+ . \]
We now set $E=\max_{0\le r<L}E_r$. This maximum is finite because there are only finitely many residue
classes modulo \(L\). Taking \(q=m-1\), we obtain, for \(m\ge 1\),
\[ \sup_{j\in\mathbb Z}
d\!\left(
S_1^\ell(\PP)_j,
S_2^L\bigl(S_1^{\ell-L}(\PP)\bigr)_j
\right) \le E\eta^m . \]
Therefore, using the contractivity of \(S_2\), we get
\begin{equation} \label{eqn:delta_S1P}
    \delta(S_1^\ell(\PP))
\le \mu\,\delta(S_1^{\ell-L}(\PP)) + 2E\eta^m,
\qquad m=\lfloor \ell/L\rfloor .
\end{equation}
Iterating \eqref{eqn:delta_S1P} gives
\[ \delta(S_1^{mL+r}(\PP)) \le \mu^m\delta(S_1^r(\PP)) +
2E\sum_{q=1}^{m}\mu^{m-q}\eta^q . \]
Let \(\alpha=\max\{\mu,\eta\}<1\). Since
\(\mu^{m-q}\eta^q\le \alpha^m\), we conclude that
\begin{equation} \label{eqn:W_recursive_bound} 
    \delta(S_1^\ell(\PP)) \le \alpha^{\lfloor \ell/L\rfloor} \left( \delta(S_1^r(P))+2E\lfloor \ell/L\rfloor
\right),
\end{equation}
where \(r=\ell-L\lfloor \ell/L\rfloor\).}

Let $t \in [t_{2j}^{\ell+1}, t_{2j+2}^{\ell+1})$, for some $j \in \mathbb{Z}$. As in~\eqref{eqn:proof_main_ineq}, we obtain also here,
\begin{multline} \label{eqn:W_proof_main}
        d \left( \pa{A_w,Q^{\ell+1}; t},\pa{A_w,Q^\ell ; t} \right) \le  
    d \left( \pa{A_w,Q^{\ell+1}; t}, S_1^{\ell+1}(\PP)_{2j} \right) \\ 
    + d \left( S_1^{\ell+1}(\PP)_{2j}, S_1^{\ell}(\PP)_{j} \right) + d \left( S_1^{\ell}(\PP)_{j} , \pa{A_w,Q^{\ell}; t} \right)
\end{multline}
The first and third terms on the RHS of~\eqref{eqn:W_proof_main} are bounded as in~\eqref{eqn:bound_term1}. Then, we use~\eqref{eqn:W_recursive_bound} to get
\begin{align*}
    d \left( \pa{A_w,Q^{\ell+\nu}; t}, S_1^{\ell+\nu}(\PP)_{2^\nu j} \right) &\le  2^\nu \delta (S_1^{\ell+\nu} (\PP))  \\
    &\le  \alpha^{\floor{(\ell+\nu)/L}} \left(  \delta (S_1^{r_\nu} (\PP)) + 2E \floor{(\ell+\nu)/L} \right)  \\
    &\le  \alpha^{\floor{\ell/L}} \left(  \delta (S_1^{r_\nu} (\PP)) + 2E (\floor{\ell/L}+1) \right) ,   
\end{align*}
where $r_\nu=(\ell+\nu)-L\floor{(\ell+\nu)/L} \in \{0,\ldots,L-1\}$. Recall that $C_{\PP} = C_{\PP}(S_1) =\max_{r=0,\ldots,L-1} \delta (S_1^r (\PP))$ is a function of $\PP,L$ and $S_1$ but is independent of $\ell$. Note that replacing $\delta (S_1^{r} (\PP))$ by $C_{\PP}$ in~\eqref{eqn:W_recursive_bound} provides that for any $\ell \in \N$,
\begin{equation*}
    \delta (S_1^\ell (\PP)) 
    \le \alpha^{\floor{\ell/L}} \left(  C_{\PP} + 2E \floor{\ell/L} \right) .  
\end{equation*} 
Therefore, for any $\ell \in \N$,
\begin{equation} \label{eqn:Wparts12}
    d \left( \pa{A_w,Q^{\ell+\nu}; t}, S_1^{\ell+\nu}(\PP)_{2^\nu j} \right) \le  \alpha^{\floor{\ell/L}} \left( C_{\PP} + 2E + 2E \floor{\ell/L} \right) , \quad \nu=0,1.
\end{equation}
The middle term on the RHS of~\eqref{eqn:W_proof_main} is bounded by $C_{S_1} \delta (S_1^{\ell} (\PP))$ since $S_1$ is \rev{displacement-safe}. Using the bound~\eqref{eqn:W_recursive_bound} we get
\begin{align} \label{eqn:W_middleterm} 
d \left( S_1^{\ell+1}(\PP)_{2j},S_1^{\ell}(\PP)_{j} \right) & \le C_{S_1} \left( \alpha^{\floor{\ell/L}} \left(  \delta (S_1^r (\PP)) + 2E \floor{\ell/L} \right) \right) \nonumber \\ 
& \le \alpha^{\floor{\ell/L}} \left(  C_{S_1} C_{\PP} + 2 C_{S_1} E \floor{\ell/L} \right) .
\end{align}
Thus, combining~\eqref{eqn:Wparts12} and \eqref{eqn:W_middleterm} we obtain,
\begin{eqnarray*}
    d \left( \pa{A_w,Q^{\ell+1}; t},\pa{A_w,Q^\ell ; t} \right) &\le & \alpha^{\floor{\ell/L}} \left(C_{S_1} C_{\PP} + 2 (C_{\PP}+2E)+(2C_{S_1}E+4E)\floor{\ell/L}\right) \\
    &\le & (\alpha^{1/L})^\ell \left(A+B\ell\right) \, \, ,
\end{eqnarray*}
where $A = \alpha^{-1}\left(C_{\PP}(C_{S_1} + 2)+4E \right) $ and $B=2(\alpha L)^{-1}E(C_{S_1}+2)$. Since $\alpha^{1/L}<1$ and  $A,B$ are independent of $\ell$, we get the desired summability. Namely, for any $t$ 
\[ \sum_{\ell=1}^\infty d \left( \pa{A_w,Q^{\ell+1}; t},\pa{A_w,Q^\ell ; t} \right) \le \sum_{\ell=1}^\infty  \left(A+B\ell\right) (\alpha^{1/L})^\ell < \infty , 
\] 
and the proof follows. 
\end{proof}
\begin{remark} \label{rem:prox_2}
    In view of Theorem~\ref{thm:convergence_proximity_2ndType}, we observe a significant distinction between the two types of proximity. The first type of proximity results in Theorem~\ref{thm:main_proximity} guaranteeing weak convergence, while the second type of proximity ensures convergence for all $\PP \in \D$. Therefore, if there is no condition on $\delta(\PP)$ for $\PP \in \D$, then $S_1$ converges strongly in $\D$. In addition, the proof of Theorem~\ref{thm:convergence_proximity_2ndType} does not rely on a contractivity factor of $S_1$.
\end{remark}

%======================== NEW PART ========================
\subsection{Comparison with the proximity of the first type}

Following Remark~\ref{rem:prox_2}, we further highlight the differences between the two notions of proximity by providing a side-by-side comparison.  While both concepts aim to relate the behavior of one subdivision scheme to another, they differ in their formulation, implications for convergence, and their relation to smoothness. In short, the first type of proximity captures one-step similarity between schemes and ensures weak convergence. In contrast, the second type requires asymptotic closeness across all refinement levels, leading to strong convergence and a natural pathway to smoothness results. The key distinctions are summarized in Table~\ref{tab:proximity_comparison}.

\begin{table}[ht!]
\centering
\renewcommand{\arraystretch}{1.3}
%\begin{tabular}{|l|p{0.3\linewidth}|p{0.38\linewidth}|}
\begin{tabularx}{\textwidth}{|X|X|X|}
\hline
\textbf{Aspect} & \textbf{First Type Proximity}~\eqref{eqn:proximity} & \textbf{Second Type Proximity}~\eqref{eqn:2nd_type_proximity} \\
\hline
Definition & Compares two schemes after a \emph{single} refinement step 
& Compares one scheme to the other \emph{across all refinement levels} \\
\hline
Symmetry & Symmetric: $S_1$ close to $S_2$ $\;\Leftrightarrow\;$ $S_2$ close to $S_1$ & Not symmetric: $S_1$ may be in proximity to $S_2$ without the converse \\
\hline
Dependency of the constant & $C$ is independent of $\PP$ & $E$ may depend on $\PP$ \\
\hline
Convergence implication & Guarantees \emph{weak convergence} & Guarantees \emph{strong convergence} \\
\hline
Link to smoothness & Higher-order versions are more technical; linked indirectly to smoothness & Smoothness guaranteed by the order $m$: $m=0$ ensures convergence, $m=1$ ensures $C^1$ smoothness in $\mathbb{R}^n$, conjecture: $m$ ensures smoothness $C^m$ in $\mathbb{R}^n$. See the following Subsection~\ref{subsec:orderm} \\
\hline
%Typical Use & Useful when contractivity of a nonlinear scheme is hard to prove directly; compare to a known convergent scheme & Useful for both convergence and smoothness proofs; streamlines arguments by replacing two-step proofs with a single proximity criterion \\ \hline
\end{tabularx}
\caption{Comparison of the two types of proximity between two subdivision schemes in CMS.} \label{tab:proximity_comparison}
\end{table}

%======================== NEW PART ========================

\subsection{\texorpdfstring{The role of the order $m$ of the proximity}{The role of m}} \label{subsec:orderm}

The order $m=0$ of the second type proximity leads to convergence, as is shown in the previous subsection. Here, we provide a glimpse into the question of higher-order second-type proximity. Before proceeding, we note that this aspect shows another difference between the two types of proximity regarding their higher-order generalizations. Specifically, second-order first-type proximity exists in the manifold settings~\cite[Definition 3]{wallner2006smoothness}, but its form becomes more complex with increasing orders. In the second-type proximity, the higher orders have the same form as the lower ones.

The high-order proximity of the first type is linked to smoothness, and so is the high-order proximity of the second type. Therefore, it is necessary to provide a smoothness notion when discussing the case $m\ge 1$. However, defining the smoothness of curves in a CMS can be a rather challenging task, beyond the scope of this paper, and we restrict the following discussion to curves in Euclidean spaces generated by nonlinear subdivision schemes. Specifically, this section shows how the second type proximity of order  $m=1$ leads to $C^1$ limit curves in $\R^n$. Even though the following discussion is much more restrictive in terms of the space we consider here, the result we present here, to the best of our knowledge, is new.

While in the case of a linear scheme generating $C^1$ curves in $\R^n$, the proof that the limit is $C^1$ can be reduced to the proof of convergence of a related subdivision scheme, but this is not the case for a nonlinear scheme generating $C^1$ curves in $\R^n$. For such a scheme, we prove the convergence of the sequences of divided differences of the sequences of points generated by the nonlinear scheme directly. We start by introducing new notions.

\begin{definition} [New notions and their notation]\label{def:more}

The first two items define the sequence of first-order divided differences of the sequence of points generated at the $k$-th refinement level by a subdivision scheme. 

\begin{itemize}
\item
Let $\PP\in(\R^n)^\Z$. Then the difference operator, denoted by $\Delta$, is defined by
, 
\begin{equation*}
(\Delta \PP)_{j} =  p_{j+1}-p_{j}, \quad j\in\Z, \quad \PP \in (\R^n)^\Z.    
\end{equation*}
\item
Let $\PP^{[k]}\in(\R^n)^\Z,$ be the sequence of points generated at the $k$-th refinement level by a subdivision scheme.
For simplicity of presentation, we assume that the underlying parameterization satisfies \newline $(\Delta T^k)_{j} = 2^{-k}$ for all $j\in\Z$ (This assumption holds for the primal and dual parameterizations ).Thus, $\Delta \PP^{[k]} / 2^{-k}$ represents the sequence of first-order divided differences at the $k$th refinement level. 
\item[$\bullet$]
 In $\R^n$, we use the metric defined by the Euclidean norm. This norm is denoted for simplicity by $\norm{\cdot}$. With this notation 
\[ \delta (\PP) = \max_{i \in \Z}\norm{\PP_{i+1}-\PP_{i}}  . \]
\item
 A subdivision scheme is $C^1$ in $ \D\subseteq \Omega^\Z$  if it converges and generates $C^1$ limit curves from any initial data $\PP\in \D$. 
\item
 A subdivision scheme $S$ has a \textbf{first-order}  contractivity factor $\mudelta \in (0,1)$ and a {\bf first order} \rev{contractivity level} $L_1\in\N$ in $\D\subset\R^n $ if, \rev{for every \(k\in\mathbb Z_+\)
and every level-\(k\) sequence \(\PP^k\in \D\),
\begin{equation} \label{eqn:contractivity_diff}
     \delta\!\left(
\frac{\Delta(S^{L_1}(\PP^k))}{2^{-(k+L_1)}}
\right)
\le
\mu_1
\delta\!\left(
\frac{\Delta \PP^k}{2^{-k}}
\right). 
     \end{equation} } \vspace{-1em}
\item 
If $S$ is contractive, we term the contractivity factor and the \rev{contractivity level} as \textbf{zero-order}, and denote them by $\mu_0$ and $L_0$ respectively.
\end{itemize}
\end{definition}
 
It is worth noting that a subdivision scheme $S$ with a first-order contractivity factor, which also satisfies the \rev{displacement-safe} property, is $C^1$. This follows from the convergence of the sequences of the first-order divided differences whose limit is the derivative of the limit of $S$, see~\cite{dyn1992subdivision_book}. 

In the following, by showing the convergence of the sequence of first-order divided differences of $S_1^k\PP^0$, we prove that the proximity of the second type of order $1$ leads to a limit curve that is (at least) $C^1$.
\begin{thm} \label{thm:C1_m1}
Consider $\R^n$ endowed with the metric $d(x,y) =\norm{x-y}$, where the norm is any norm in $\R^n$. Let $S_2$ be a subdivision scheme defined on $\R^n$ that satisfies~\eqref{eqn:contractivity_diff} with a contractivity factor $\mudelta$ and a \rev{contractivity level} $L_1$ in $\D \subset \big( \R^n \big)^\Z$. Let $S_1$ be a subdivision scheme defined on $\R^n$ satisfying:
    \begin{enumerate}
        \item % (parameterization)
        $S_1$ is defined relative to the same parameter sequences as $S_2$.
        \item % (closeness)
        For all $\PP \in \D$,  $S_1(\PP) \in \D$.
        \item %(Proximity) 
        The scheme \(S_1^{L_1}\) is in proximity of the second type of order \(1\) with \(S_2^{L_1}\), with proximity factor $0<\eta<2^{-L_1}$.
        \item 
        \(S_1\) is displacement-safe in \(\D\).
    \end{enumerate}
    Then, $S_1$ converges to a $C^1$ limit from any $\PP\in \D$.
\end{thm}
\begin{proof}
    The proof follows similar lines to those of the proof of Theorem~\ref{thm:convergence_proximity_2ndType}. First, for \(\ell \ge L_1\), the \rev{triangle inequality gives},
\begin{multline} \label{eqn:proximity_diff}
    \norm{\left(\Delta (S_1^\ell (\PP)) / 2^{-\ell}  \right)_j - \left(\Delta ( S_1^\ell ( \PP ))/ 2^{-\ell}  \right)_{j+1} }
   \le   \\
  \norm{\left(\Delta (S_1^{L_1} S_1^{\ell-L_1} (\PP)) / 2^{-\ell}  \right)_j- \left(\Delta (S_2^{L_1} S_1^{\ell-L_1} (\PP)) / 2^{-\ell}  \right)_j } \nonumber  \\
  \hspace{80pt}  +   \norm{ \left(\Delta (S_2^{L_1} S_1^{\ell-{L_1}} (\PP)) / 2^{-\ell}  \right)_j - \left(\Delta (S_2^{L_1} S_1^{\ell-{L_1}} (\PP)) / 2^{-\ell}  \right)_{j+1} } \nonumber \\
    +  \norm{\left(\Delta (S_2^{L_1} S_1^{\ell-{L_1}} (\PP)) / 2^{-\ell}  \right)_{j+1} - \left(\Delta (S_1^{L_1} S_1^{\ell-{L_1}} (\PP)) / 2^{-\ell}  \right)_{j+1} } . 
\end{multline} 
\rev{Let
\[
    \ell=mL_1+r,\qquad
    m=\left\lfloor\frac{\ell}{L_1}\right\rfloor,\qquad
    r\in\{0,\ldots,L_1-1\}.
\]
For each such \(r\), set $\PP^r=S_1^r(\PP)$. Since \(S_1(\D)\subset \D\), we have \(\PP^r\in \D\). Applying the second-type proximity estimate for \(S_1^{L_1}\) and \(S_2^{L_1}\) to the finitely many initial sequences \(\PP^r\), \(r=0,\ldots,L_1-1\), there are constants \(E_r>0\) such that
\[
\sup_{j\in\mathbb Z}
\left\|
S_1^{L_1}\bigl((S_1^{L_1})^q(\PP^r)\bigr)_j
-
S_2^{L_1}\bigl((S_1^{L_1})^q(\PP^r)\bigr)_j
\right\|
\le
E_r\eta^{q+1},
\qquad q\in\mathbb Z_+ .
\]
We set $E=\max_{0\le r<L_1}E_r$. Taking \(q=m-1\), we obtain, for \(m\ge 1\),
\[
\sup_{j\in\mathbb Z}
\left\|
S_1^\ell(\PP)_j
-
S_2^{L_1}\bigl(S_1^{\ell-L_1}(\PP)\bigr)_j
\right\|
\le
E\eta^m .
\]
Consequently,
\[
\begin{aligned}
\left\|
\left(
\frac{\Delta S_1^\ell(P)}{2^{-\ell}}
\right)_j
-
\left(
\frac{\Delta S_2^{L_1}(S_1^{\ell-L_1}(P))}{2^{-\ell}}
\right)_j
\right\|  
\le
2^{\ell+1}E\eta^m
=
2^{r+1}E(2^{L_1}\eta)^m
\le
2^{L_1+1}E\beta^m,
\end{aligned}
\]
where $\beta=2^{L_1}\eta<1$. The same estimate holds with \(j\) replaced by \(j+1\). Therefore, using
the first-order contractivity of \(S_2\), we obtain
\[
\delta\left(
\frac{\Delta S_1^\ell(P)}{2^{-\ell}}
\right)
\le
\mu_1
\delta\left(
\frac{\Delta S_1^{\ell-L_1}(P)}{2^{-(\ell-L_1)}}
\right)
+
2^{L_1+2}E\beta^m .
\]
Iterating this estimate over the fixed residue class \(r\) gives
\[
\delta\left(
\frac{\Delta S_1^{mL_1+r}(P)}{2^{-(mL_1+r)}}
\right)
\le
\mu_1^m
\delta\left(
\frac{\Delta S_1^r(P)}{2^{-r}}
\right)
+
2^{L_1+2}E
\sum_{q=1}^{m}
\mu_1^{m-q}\beta^q .
\] }
Let $\alpha=\max\{\mu_1,\beta\}<1$. Since \(\mu_1^{m-q}\beta^q\le \alpha^m\), we conclude that
\[
\delta\left(
\frac{\Delta S_1^\ell(P)}{2^{-\ell}}
\right)
\le
\alpha^{\lfloor \ell/L_1\rfloor}
\left(
\delta\left(
\frac{\Delta S_1^r(P)}{2^{-r}}
\right)
+
2^{L_1+2}E
\left\lfloor\frac{\ell}{L_1}\right\rfloor
\right),
\]
where \(r=\ell-L_1\lfloor \ell/L_1\rfloor\). Thus, the rest of the proof of Theorem~\ref{thm:convergence_proximity_2ndType} is valid for 
\[ Q^\ell=\{(t_j^\ell,(\Delta (S_1^\ell (\PP)) / 2^{-\ell})_j\}_{j\in \Z}, \]
namely, it is a Cauchy sequence. Therefore, there exists a continuous limit to the sequence of sequences of divided differences of first order. Moreover, by the standard implication recalled above, the first-order contractivity of \(S_2\) implies its zero-order contractivity. Hence \(S_2\) is contractive in the sense required by Theorem~\ref{thm:convergence_proximity_2ndType}. Since proximity of order \(1\) implies proximity of order \(0\), and since \(S_1\) is displacement-safe in \(\D\), Theorem~\ref{thm:convergence_proximity_2ndType} yields the convergence of \(S_1\). 
\end{proof}

As a first conclusion from the proof of Theorem~\ref{thm:C1_m1}, we see that the convergence is guaranteed since the proximity for \( m = 0 \) follows from the case \( m = 1 \). Moreover, the convergence of the divided differences to the derivative implies that the limit curve is differentiable.  %and its derivative is the limit of the derivatives.

To conclude, we conjecture the general order case:
\begin{conj}
Consider $\R^n$ endowed with the metric $d(x,y) =\norm{x-y}$, where the norm is any norm in $\R^n$. Let $S_2$ be a subdivision scheme defined on $\R^n $, \rev{satisfying, for \(q=0,\ldots,m\), some \(\mu_q\in(0,1)\) and \(L_q\in\mathbb N\), and for every \(k\in\mathbb Z_+\) and every level-\(k\) sequence \(\PP^k\in \D\),
\begin{equation}
\delta\!\left(
\frac{\Delta^q(S_2^{L_q}(\PP^k))}{2^{-q(k+L_q)}}
\right)
\le
\mu_q
\delta\!\left(
\frac{\Delta^q \PP^k}{2^{-qk}}
\right).
\end{equation}}
    Let $S_1$ be a subdivision scheme defined on $\R^n$ and satisfying there:
    \begin{enumerate}
        \item % (parameterization)
        $S_1$ is defined relative to the same parameter sequences as $S_2$.
        \item %(Proximity) 
        The scheme $S_1^{L_m}$ is in \textbf{proximity of the second type of order} $\mathbf{m}$ with $S_2^{L_m}$ in $\R^n$.
    \end{enumerate}
    Then, $\mathbf{S_1}$ \textbf{converges to a} $\mathbf{C^m}$ limit.
\end{conj}
We leave the study of this conjecture to future work.

\subsection {Application of the second type proximity of order \texorpdfstring{$m=0$ and $m=1$}{m=0 and m=1} } \label{subsec:application_second_prox}

In Section~\ref{subsec:app_convergence}, we demonstrated the convergence of elementary schemes with intrinsic averages. However, these schemes do not produce $C^1$ limits; for instance, the limit of the linear version of the first divided differences exhibits a jump discontinuity at the input data points. In this section, we show that by applying an additional round of averaging in each refinement, we can guarantee a $C^1$ limit using the second type of proximity.

We consider $\Omega$ to be a Riemannian manifold embedded in $\R^n$. Therefore, we refer to the $C^1$ smoothness of any curve in $\Omega$ with respect to the embedded Euclidean space. Here, the average $A_w$ is the geodesic average, which is intrinsic.

In~\cite{dyn2017global}, we introduce an extension of the linear spline schemes (see Section~\ref{sec:perlimin}), which is a method for evaluating spline subdivision schemes through a series of averaging rounds. Here, we focus on a specific type of refinement, where an additional round of uniform averaging is incorporated into the elementary scheme. Namely, instead of the elementary scheme~\eqref{eqn:elementary_scheme}, the corresponding refinement rules are,
\begin{equation} \label{eqn:additional_round}
 S(\PP)_{2j}=A_{\omega}\left(p_j,A_\frac{1}{2}(p_j,p_{j+1})\right) \quad S(\PP)_{2j+1}=A_{\omega}\left(A_\frac{1}{2}(p_j,p_{j+1}),p_{j+1} \right) ,\quad j\in\Z.
\end{equation}
with a fixed $\omega \in (0,1)$. Note that the spline subdivision scheme in $\R^n$~\eqref{eqn:cc} uses $\omega=\frac{1}{2}$, which leads to the well-known Chaikin’s corner-cutting algorithm~\cite{chaikin1974algorithm}.

\rev{The main result of this section demonstrates how the second type of proximity can be employed to establish strong $C^1$ smoothness. It is worth noting that this does not yield a fundamentally new $C^1$ result. By \cite{wallner2005convergence}, $C^1$ smoothness can already be obtained in a weak sense. Together with the strong convergence from \cite{dyn2017global}, the two results guarantee that the scheme is $C^1$. In other words, the contribution of this example, aside from illustrating the applicability of the case $m=1$, lies not in providing a new smoothness result, but in simplifying the proof: it reduces a two-step argument (convergence plus proximity of the first type) to a single proximity criterion, of the second type. }

\begin{thm} \label{thm:C1_manifold}
    Let $\Omega \subset \R^n$ be a complete Riemannian manifold and let $\D \subseteq \Omega^\Z$. Assume that for $\PP\in\D$, any pair of consecutive points has a unique geodesic curve that connects them. Then, the scheme~\eqref{eqn:additional_round} with the geodesic average, converges to a $C^1$ limit.
\end{thm}
The above result is obtained using the proximity of the second type of order 1. Recall that the conditions of $m=1$ imply both convergence and $C^1$ smoothness, since it includes the case $m=0$. The proof is based on the following two observations stated in the next two lemmas. 

In this example, the norm is the Euclidean norm and the average $A_{\omega}(p_0,p_1)$ is the Riemannian intrinsic average. Namely, $A_{\omega}(p_0,p_1)$ is the point along the unique geodesic connecting $p_0$ and $p_1$ that lies in a distance $\omega \, d(p_0,p_1)$ from $p_0$. One can verify that due to the properties of the Riemannian geodesic, this is indeed an intrinsic average and it satisfies the metric property~\eqref{defd=metric property}.
\begin{lemma} \label{lemma:dist}
There exists a constant $C$, independent of the points $p_0,p_1$ such that
    \[ \norm{A_{\omega}(p_0,p_1) - \left((1-\omega)p_0+ \omega p_1\right)}  \le C (\omega+\omega^2) \norm{p_0-p_1}^2  , \quad \omega \in (0,1) . \]
\end{lemma} 
For proof see~\cite[Lemma 4]{wallner2005convergence}. The next lemma directly proves Theorem~\ref{thm:C1_manifold}.
\begin{lemma}
    The scheme~\eqref{eqn:additional_round} satisfies the conditions of Theorem~\ref{thm:C1_m1}, where the proximity of order $m=1$ is to its linear counterpart in $\R^n$, that is~\eqref{eqn:additional_round} where $A_{\omega}$ is the arithmetic average.
\end{lemma}
\begin{proof}
   Denote by $S$ the scheme of Theorem~\ref{thm:C1_manifold} and by $S_{\text{lin}}$ its linear counterpart in $\R^n$. The scheme $S$ satisfies the contractivity~\eqref{eqn:contractivity_elementary} of the elementary scheme since it generates the two points in the next level in~\eqref{eqn:additional_round} on the geodesic between $p_j$ and $p_{j+1}$. This is, in fact, a special case of~\cite[Theorem 3.1]{dyn2017global}. Namely, if $d$ is the geodesic distance over $\Omega$, we have that $\delta(S^{\ell}(\PP))\le (\frac{1}{2})^{\ell} \delta(\PP)$, $\ell \in \N$. Then, using Lemma~\ref{lemma:dist} and the fact that the Euclidean distance always bounds from below the geodesic distance, we have 
   \begin{eqnarray*}
       \norm{S(\PP^{\ell})_{2j}-S_{\text{lin}}(\PP^{\ell})_{2j}} & = & \resizebox{0.7\hsize}{!}{$\norm{ A_{\omega}\left((\PP^{\ell})_j,A_\frac{1}{2}((\PP^{\ell})_j,(\PP^{\ell})_{j+1})\right)-  \left[(1-\omega)(\PP^{\ell})_j+ \omega \left(\frac{1}{2}(\PP^{\ell})_j+\frac{1}{2}(\PP^{\ell})_{j+1}\right) \right]}$} \\
       & \le & C (\omega+\omega^2) \norm{\frac{1}{2}((\PP^{\ell})_j-(\PP^{\ell})_{j+1})}^2 \\
       & \le & (2C) \,  \frac{1}{4} \,  d((\PP^{\ell})_{j},(\PP^{\ell})_{j+1})^2 \\
       & \le & (2C) \,  \frac{1}{4} \, \delta(\PP^{\ell})^2  \le (2C) \,  \frac{1}{4} \, \left( (\frac{1}{2})^{\ell}\delta(\PP^0)\right)^2 = \left(2C \delta(\PP^0)\right) \,  \big(\frac{1}{4}\big)^{\ell+1} =  E \eta^{\ell+1} .
   \end{eqnarray*}
Here $E = 2C \delta(\PP^0)$ and $\eta=\frac{1}{4}<\frac{1}{2}$. A similar argument applies to $\norm{S(\PP^{\ell})_{2j+1}-S_{\text{lin}}(\PP^{\ell})_{2j+1}}$, yielding a proximity of order $m=1$ between $S$ and $S_{\text{lin}}$. The other conditions of Theorem~\ref{thm:C1_m1} are met straightforwardly.  
\end{proof}

% ====================================
\rev{
\section{Non-stationary subdivision schemes} \label{sec:nonstationary}

We now extend part of the framework developed in the previous sections to non-stationary subdivision schemes and refer to the schemes in those sections as ``stationary.'' Throughout this section, we use the primal parameterization. Thus, after \(N\) binary refinement steps, the associated parameter sequence is \(T^N=2^{-N}\mathbb Z\).

We begin with a simple example. Let $\{A^{[j]}_\omega\}_{j\geq 0}$ be a sequence of binary averages on $\Omega$. For each $j \geq 0$, define a refinement operator $S_j : \Omega^{\mathbb{Z}} \to \Omega^{\mathbb{Z}}$ by
\[
S_j(\PP)_{2i} = \PP_i, \qquad 
S_j(\PP)_{2i+1} = A^{[j]}_{1/2}(\PP_i, \PP_{i+1}), \quad i \in \mathbb{Z}.
\]                                   
The corresponding parameterization is the primal parametrization, as in the case of the stationary elementary schemes. Each operator $S_j$ has the same local structure, differing only in the choice of the average. If the average is fixed, namely $A^{[j]}_\omega \equiv A_\omega$ for all $j \in \Z_+$, then all refinement operators coincide, and the resulting scheme is the stationary elementary scheme. However, if the averages vary with the refinement level, then the refinement rule changes from step to step, and the subdivision process is naturally described by a sequence of refinement operators. This leads to a non-stationary subdivision scheme, in the sense that the refinement applied at each level depends on the level itself.

\subsection{Definition and convergence of non-stationary subdivision schemes}

We begin with a new definition of the admissibility of a refinement operator.
\begin{definition}(Admissible refinement operator relative to refinement level \(L\))
\label{def:admissible}
Let \(S\) be a refinement operator defined on sequences in
\[
    \D\subset \{\PP\in\Omega^{\Z}:\delta(\PP)<\infty\}.
\]
Then, \(S\) is admissible relative to refinement level \(L\in\N\) if
\(S^m(\D)\subset\D\), for \(m=1,\ldots,L\), and if the block \(S^L\)
satisfies the following two estimates. First,
\begin{equation} \label{eqn:mu}
\mu^{(L)}
=
\sup_{\PP\in\D,\,\delta(\PP)>0}
\frac{\delta(S^L(\PP))}{\delta(\PP)}
<\infty .
\end{equation}
Second, there exists a constant \(C_S^{(L)}<\infty\) such that
\begin{equation} \label{eqn:block_DS}
\sup_{i\in\Z}
d\bigl((S^L(\PP))_{2^L i},\PP_i\bigr)
\le
C_S^{(L)}\delta(\PP),
\qquad \PP\in\D .
\end{equation}
We call \(\mu^{(L)}\) the refinement factor of \(S\) associated with refinement level \(L\), and \(C_S^{(L)}\) the displacement-safe constant
of the block \(S^L\).
\end{definition}

Note that if the refinement factor of the above definition is less than one, then $\mu^{(L)}$ is a contractivity factor with contractivity level $L$. Next, we present the definition of a non-stationary subdivision scheme.
\begin{definition}(Non-stationary subdivision scheme)
\label{def:nns_subdivision}
Let $\s= \{ S_j \}_{j=0}^\infty$ be a sequence of admissible refinement operators with respect to a corresponding sequence of refinement levels $ \LL = \{ L_j \}_{j=0}^\infty$. For each refinement operator \(S_j\), we denote by \(\D_j\) its domain and by \(C_{S_j}^{(L_j)}\) the displacement-safe constant of the block \(S_j^{L_j}\) in the sense of~\eqref{eqn:block_DS}. We further assume that
\[
    \D^\ast=\bigcap_{j\in\Z_+}\D_j
\]
is nonempty and invariant under all refinement blocks. Namely,
\[
    S_j^m(\D^\ast)\subset \D^\ast,
    \qquad m=1,\ldots,L_j,\quad j\in\Z_+ .
\] 
We denote the corresponding non-stationary refinement process by
\[
    \s^{\LL}=\{S_j^{L_j}\}_{j\in\Z_+}.
\]
Whenever the following limit exists, we write
\begin{equation*}
\s^{\LL} (\PP)=\lim_{k\rightarrow \infty}S_k^{L_k} S_{k-1}^{L_{k-1}}
\cdots S_1^{L_1}S_0^{L_0}(\PP),\quad \PP\in \D^\ast .
\end{equation*}
\end{definition}

Definition~\ref{def:nns_subdivision} highlights the inherent flexibility of non-stationary schemes, allowing for a wide range of choices in both sequences $\s$ and $ \LL$. Note that in the case that $L_j=1$, for all $j\in \Z_+$, we obtain a well-recognized version of the classical definition of a non-stationary scheme. We recall~\eqref{eqn:mu} and conclude that 
\begin{equation} \label{eqn:cont_non-stationary}
 \delta (S_j^{L_j} (\PP)) \le \mu_j^{(L_j)} \delta (\PP) .
\end{equation}
Then, repeatedly applying the refinement by the non-stationary subdivision scheme $\s^{\LL}=\{ S_j^{L_j} \}_{j=0}^\infty$, we obtain:
\begin{equation} \label{eqn:repeated_cont_non-stationary}
\delta (S^{L_{k}}_{k}S^{L_{k-1}}_{k-1}S^{L_{k-2}}_{k-2}\cdots S^{L_1}_{1} S^{L_0}_{0} (\PP)) \le \left(\Pi_{j=0}^k \, \mu_j^{(L_j)} \right)\delta (\PP)
\end{equation}
For a given set of refinement levels $\{ L_{j} \}_{j=0}^\infty$, we introduce the following notation for the finite product of refinement factors and finite composition of refinements:
\begin{equation} \label{eqn:Lk}
\pi^{(\LL)_{0}^k} = \Pi_{j=0}^k \, \mu_j^{(L_j)}  \qquad \text{and} \qquad
 \s^{(\LL)_{0}^k} =S^{L_{k}}_{k}S^{L_{k-1}}_{k-1}S^{L_{k-2}}_{k-2}\cdots S^{L_1}_{1} S^{L_0}_{0}  . 
\end{equation}
Also, denote by
\[
    Q^{(\LL)_0^k}
    =
    \left(T^{N_k},\s^{(\LL)_0^k}(\PP)\right)
\]
the refinement of \(\PP\) by the finite application of the non-stationary scheme \(\s^{\LL}\), together with its associated primal parameters.

The following theorem is analogous to the convergence theorem for stationary schemes. Here, $A_\omega$ is an average defined on $\Omega$. 
\begin{thm}[non-stationary convergence] 
\label{thm:convergence_non-stationary}
Let $\s^{\LL} = \{ S_j^{L_j}\}_{j\in\Z_+} $ be a non-stationary subdivision scheme. Assume the following holds:
\begin{enumerate}
    \item The refinement levels, the block displacement-safe constants, and the refinement factors are uniformly bounded. Namely, there exist
    \(L^\ast\in\N\) and \(C^\ast,\mu^\ast\in[0,\infty)\) such that
    \[
        L_j\le L^\ast,\qquad
        C_{S_j}^{(L_j)}\le C^\ast,\qquad
        \mu_j^{(L_j)}\le \mu^\ast,
        \qquad j\in\Z_+ .
    \]
    \item The sequence of products of refinement factors $\pi^{(\LL)_{0}^k}$ is summable, 
\begin{equation}
\sum_{k=0}^\infty \pi^{(\LL)_{0}^k}   = \sum_{k=0}^\infty \left(\Pi_{j=0}^k \, \mu_j^{(L_j)}\right) < \infty .
\end{equation} 
\end{enumerate}
Then, for every initial data \(\PP^0\in\D^\ast\), the sequence of piecewise average interpolants
\[
    \left\{
    \pa{A_w,Q^{(\LL)_0^k};t}
    \right\}_{k\in\Z_+}
\]
converges uniformly in \(t\in\R\) to a continuous limit
\(Q^\infty(t)\in\Omega\).
\end{thm} 
\begin{proof}
The proof follows the lines of the proof of Theorem~\ref{thm:convergence}.
We show that the sequence of \(\Omega\)-valued functions
\[
    \left\{\pa{A_w,Q^{(\LL)_0^k};t}\right\}_{k\in\Z_+}
\]
is Cauchy with respect to \(d_\infty\).

For brevity, set 
\[ \PP^{[k]}=\s^{(\LL)_0^k}(\PP)\quad \text {and} \quad 
    N_k=\sum_{\nu=0}^k L_\nu .
\] 
Then, \(Q^{(\LL)_0^k}=(T^{N_k},\PP^{[k]})\). By
\eqref{eqn:repeated_cont_non-stationary},
\begin{equation} \label{eqn:delta_Pk_ns}
    \delta(\PP^{[k]})
    \le
    \pi^{(\LL)_0^k}\delta(\PP).
\end{equation}

Now, we estimate the distance between two consecutive interpolants. Fix \(k\in\Z_+\) and \(t\in\R\). Let \(i\in\Z\) be such that $t\in[t_i^{N_k},t_{i+1}^{N_k})$. Since \(N_{k+1}=N_k+L_{k+1}\), this interval is subdivided at the next non-stationary block into \(2^{L_{k+1}}\) subintervals. Hence, there exists
\(r\in\{0,\ldots,2^{L_{k+1}}-1\}\) such that
\[
    t\in
    [t_{2^{L_{k+1}}i+r}^{N_{k+1}},
     t_{2^{L_{k+1}}i+r+1}^{N_{k+1}}).
\]
By the triangle inequality,
\begin{align}
\label{eqn:proof_main_ineq_ns}
d \left( \pa{A_w,Q^{(\LL)_0^{k+1}};t},
          \pa{A_w,Q^{(\LL)_0^k};t} \right)
&\le
d \left( \pa{A_w,Q^{(\LL)_0^{k+1}};t},
          \PP^{[k+1]}_{2^{L_{k+1}}i} \right) \nonumber\\
&\quad+
d\left( \PP^{[k+1]}_{2^{L_{k+1}}i},
         \PP^{[k]}_i \right) \nonumber\\
&\quad+
d\left( \PP^{[k]}_i,
         \pa{A_w,Q^{(\LL)_0^k};t} \right).
\end{align}

We now bound the three terms on the right-hand side. For the first term,
using the boundedness property of the average and then summing over
neighboring points in the refined sequence, we obtain
\begin{align}
d \left( \pa{A_w,Q^{(\LL)_0^{k+1}};t},
          \PP^{[k+1]}_{2^{L_{k+1}}i} \right)
&\le
d \left( \pa{A_w,Q^{(\LL)_0^{k+1}};t},
          \PP^{[k+1]}_{2^{L_{k+1}}i+r} \right)  \nonumber\\
&\quad+
d\left(\PP^{[k+1]}_{2^{L_{k+1}}i+r},
        \PP^{[k+1]}_{2^{L_{k+1}}i}\right) \nonumber\\
&\le
\delta(\PP^{[k+1]})+r\,\delta(\PP^{[k+1]}) \nonumber\\
&\le
2^{L_{k+1}}\delta(\PP^{[k+1]}).
\label{eqn:bound_term1_ns}
\end{align}
For the second term, we use the block displacement-safe estimate
\eqref{eqn:block_DS} for \(S_{k+1}^{L_{k+1}}\):
\begin{equation}\label{eqn:bound_term2_ns}
d\left( \PP^{[k+1]}_{2^{L_{k+1}}i}, \PP^{[k]}_i \right)
\le
C_{S_{k+1}}^{(L_{k+1})}\delta(\PP^{[k]})
\le
C^\ast\delta(\PP^{[k]}).
\end{equation}
For the third term, the boundedness property of the average gives
\begin{equation} \label{eqn:bound_term3_ns}
d\left( \PP^{[k]}_i,\pa{A_w,Q^{(\LL)_0^k};t} \right)
\le
\delta(\PP^{[k]}).
\end{equation}
Combining \eqref{eqn:proof_main_ineq_ns}--\eqref{eqn:bound_term3_ns},
and using \eqref{eqn:mu}, \eqref{eqn:delta_Pk_ns}, and the uniform
bounds in the assumptions, we get
\begin{align*}
d^k(t)
&:=
d \left( \pa{A_w,Q^{(\LL)_0^{k+1}};t},
          \pa{A_w,Q^{(\LL)_0^k};t} \right)  \\
&\le
2^{L_{k+1}}\delta(\PP^{[k+1]})
+
(C^\ast+1)\delta(\PP^{[k]}) \\
&\le
\left(2^{L_{k+1}}\mu_{k+1}^{(L_{k+1})}+C^\ast+1\right)
\delta(\PP^{[k]}) \\
&\le
\left(2^{L^\ast}\mu^\ast+C^\ast+1\right)
\pi^{(\LL)_0^k}\delta(\PP).
\end{align*}
Set
\[
    B=2^{L^\ast}\mu^\ast+C^\ast+1 .
\]
Then, for every \(m\in\N\),
\begin{align*}
d_\infty\left(
\pa{A_w,Q^{(\LL)_0^{k+m}};t},
\pa{A_w,Q^{(\LL)_0^k};t}
\right)
&\le
\sum_{\ell=k}^{k+m-1}
d_\infty\left(
\pa{A_w,Q^{(\LL)_0^{\ell+1}};\cdot},
\pa{A_w,Q^{(\LL)_0^\ell};\cdot}
\right) \\
&\le
B\,\delta(\PP)
\sum_{\ell=k}^{k+m-1}\pi^{(\LL)_0^\ell}.
\end{align*}
Since \(\sum_{\ell=0}^\infty\pi^{(\LL)_0^\ell}<\infty\), the tails of this
series tend to zero as \(k\to\infty\), uniformly in \(m\). Therefore
\[
    \left\{\pa{A_w,Q^{(\LL)_0^k};t}\right\}_{k\in\Z_+}
\]
is a Cauchy sequence with respect to \(d_\infty\).

Since \(X=(\Omega,d)\) is complete, this Cauchy sequence has a uniform
limit \(Q^\infty:\R\to\Omega\). Finally, each function
\(\pa{A_w,Q^{(\LL)_0^k};t}\) is continuous in \(t\), by the continuity of the average with respect to the averaging parameter. Hence, the uniform limit \(Q^\infty\) is continuous.
\end{proof}

\begin{remark}[On proximity in the non-stationary setting]
A natural question is whether the proximity arguments of Sections~\ref{sec:proximity1} and~\ref{sec:proximity2} can be extended to the non-stationary setting. A direct approach would compare each block \(S_j^{L_j}\) with a fixed reference block, or with a reference sequence of blocks, using estimates that are uniform in the level \(j\). Such estimates can serve as a tool for verifying the assumptions of Theorem~\ref{thm:convergence_non-stationary}. Namely, the uniform control of the block refinement factors \(\mu_j^{(L_j)}\), the block displacement constants \(C_{S_j}^{(L_j)}\), and the summability of the products \(\prod_{\nu=0}^k\mu_\nu^{(L_\nu)}\).

More general non-stationary proximity conditions, in which the proximity measure or reference schemes vary with the level, would require tracking their interaction with the products \(\prod_{\nu=0}^k \mu_\nu^{(L_\nu)}\), and, when the refinement levels are not uniformly bounded, also with the factors \(2^{L_k}\). Although such a formulation appears less transparent as a practical convergence tool than the direct refinement-factor criterion of Theorem~\ref{thm:convergence_non-stationary}, it may provide a useful framework for treating non-stationary schemes whose convergence is more naturally understood through comparison with simpler or better-studied schemes. In particular, such an approach may lead to refined convergence criteria and possibly to smoothness results for non-stationary schemes. We leave the development of this proximity-based non-stationary theory for future work.
\end{remark}
} %rev of This section

% ====================================

\section*{Acknowledgment} 
We thank Hofit Ben Zion Vardi for her contribution to the conception of this project in its early stages. Nir Sharon is partially supported by the NSF-BSF award 2024791, the BSF award 2024266, and the DFG award 514588180. 

% ====================================

\bibliographystyle{plain} 
\bibliography{revised_bib}

% ====================================

\end{document}